\documentclass{amsart}
\RequirePackage{fix-cm}
\RequirePackage{amsmath}
\RequirePackage{amssymb}
\RequirePackage{amsfonts}
\RequirePackage{xcolor}
\RequirePackage{centernot}
\RequirePackage{enumerate}
\RequirePackage{faktor}

\usepackage{graphicx}
\newcommand{\R}{\mathbb{R}}
\newcommand{\C}{\mathbb{C}}
\newcommand{\N}{\mathbb{N}}

\newcommand{\T}{\mathbb{T}}

\newcommand{\Pc}{\mathcal{P}}

\newcommand{\Cheb}{\mathrm{Cheb}}
\newcommand{\Oc}{\mathcal{O}}
\newcommand{\ee}{\varepsilon}
\newcommand{\lo}{\longrightarrow}
\newcommand{\li}{\left}
\newcommand{\re}{\right}
\newcommand{\mi}{\,\,\big|\,\,}

\newcommand{\rank}{\mathrm{rank}}

\newcommand{\GR}{\mathrm{Gr}}

\newtheorem{theo}{Theorem}[section]
\newtheorem{theorem}{Theorem}[section]
\newtheorem{lemma}[theorem]{Lemma}
\newtheorem{corollary}[theorem]{Corollary}
\theoremstyle{definition}
\newtheorem{definition}[theorem]{Definition}

\theoremstyle{remark}
\newtheorem{remark}[theorem]{Remark}

\numberwithin{equation}{section}

\newtheorem{experiment}{Experiment}
\newtheorem{quest}{Question}

\usepackage{pdfpages}

\begin{document}
% \includepdf[pages={1,2}]{Cover2-signed.pdf}

\title[Multivariate Interpolation in Unisolvent Nodes]{Multivariate Interpolation in Unisolvent Nodes \\ Lifting the Curse of Dimensionality}

\author[M. Hecht]{Michael Hecht}

\author[K.~Gonciarz]{Krzysztof Gonciarz}
\author[J.~Michelfeit]{Jannik Michelfeit}
% \address{Faculty of Electrical Engineering, TU Dresden, Dresden, Germany}
\author[V.~Sivkin]{Vladimir Sivkin}
% \address{Faculty of Mathematics, Lomonosov Moscow State University, Moscow, Russia}
\author[I.F.~Sbalzarini]{Ivo F.~Sbalzarini}


% \date{January 1, 2001 and, in revised form, June 22, 2001.}

% \dedicatory{This paper is dedicated to our advisors.}


\begin{abstract} We extend Newton and Lagrange interpolation to arbitrary dimensions.
The core contribution that enables this is a generalized notion of \emph{non-tensorial unisolvent nodes}, i.e., nodes on which the multivariate polynomial interpolant of a function is unique.
By validation, we reach the optimal exponential Trefethen rates for a class of analytic functions, we term  Trefethen functions.
The number of interpolation nodes required for computing the optimal interpolant depends sub-exponentially on the dimension, hence resisting the curse of dimensionality.
Based on these results, we propose an algorithm to efficiently and numerically stably solve arbitrary-dimensional interpolation problems,
with at most quadratic runtime and linear memory requirement.
\end{abstract}

\maketitle

\section{Introduction}

\label{intro}
Polynomial interpolation goes back to Newton, Lagrange, and others \cite{LIP}, and its fundamental importance for mathematics and computing is undisputed.
Interpolation is based on the fact that, in 1D, one and only one polynomial $Q_{f,n}$ of degree $n$ can interpolate a function
$f : \R \lo \R$ on $n+1$ distinct \emph{unisolvent interpolation nodes} $P_n \subseteq \R$, i.e., $Q_{f,n}(p_i) =f(p_i)$ for all $p_i \in P_n$, $0\leq i \leq n$.
This makes interpolation fundamentally different from approximation. For the latter, the famous \emph{Weierstrass Theorem} \cite{weier1}
states that any continuous function $f \in C^0(\Omega,\R)$, $\Omega=[-1,1]$, can be uniformly approximated by
polynomials in principle~\cite{weier2,glimm,weier1,stephen}. However, the Weierstrass Theorem does not require the polynomials to coincide with $f$ at all, i.e.,
it is possible that $Q_{\mathrm{Weierstrass},f,n}(x) \not = f(x)$ for all $x\in \Omega$, but still
\begin{equation}
Q_{\mathrm{Weierstrass},f,n} \xrightarrow[n \rightarrow \infty]{} f \quad \text{uniformly on}\,\,\, \Omega\,.
\end{equation}
Even though the constructive version of the Weierstrass Theorem given by Serge Bernstein \cite{bernstein1912} provides a recipe for computing such approximations it only delivers a linear convergence rate.
In 1D, however, interpolation on Chebyshev and Legendre nodes is known to avoid Runge's phenomenon for a generic class of functions and to yield exponential
approximation rates \cite{Lloyd}, which is much faster than what has been shown possible by Weierstrass-type approximations \cite{bernstein1912}.
There has therefore been much research into faster multi--dimensional ($m$D) interpolation schemes that extend Newton or Lagrange interpolation schemes to arbitrary dimensions.
Any approach that addresses this problem has to resolve \emph{Runge's phenomenon} \cite{faber,jackson1913,runge} and resist the \emph{curse of dimensionality}, thus, has to answer the following question:

\begin{quest}\label{Q1}
How to construct interpolation nodes $P_{A_{m,n}} \subseteq \Omega  =[-1,1]^m$, $m,n \in \N$, $A_{m,n}\subseteq \N^m$ and an efficient and numerically stable interpolation algorithm such that a generic class of functions
$f : \Omega  \lo \R$ can be uniformly approximated
$$Q_{f,A_{m,n}}   \xrightarrow[n \rightarrow \infty]{}  f \,, \qquad Q_{f,A_{m,n}} (q) = f(q) \,, \,\,\, \forall\, q \in P_{A_{m,n}}\,,$$ with a fast (ideally exponential) convergence
rate while keeping the number of interpolation nodes $|P_{A_{m,n}}|$ required small (ideally, sub-exponential)?
\end{quest}
Question~\ref{Q1} implicitly requires answering the underlying algebraic question of \emph{unisolvence}:

\begin{quest}\label{Q2} Given a polynomial space $\Pi=\mathrm{span}\{q_{\alpha}\}_{\alpha \in A}$, $A\subseteq \N^m$ generated by some polynomial basis.
How to construct unsisolvent interpolation nodes $P_{A} \subseteq \Omega $,  such that the polynomial interpolant
$Q_{f,A}$, with $Q_{f,A}(p) =f(p) \,, \forall p \in P_A$, of any function $f : \Omega \lo \R$ is uniquely determined in $\Pi$?
\end{quest}

While many approaches exist to extending polynomial interpolation to higher dimensions
 \cite{Bos,deBoor2,deBoor,Erb,dyn,Gasca,sauer,Guenther,2000,MultiVander,sauerL,Lloyd}, providing partial answers to Questions~\ref{Q1} and \ref{Q2},
none of them fully answers these questions.
% and thereby extends the full power of Newton, Lagrange, or other 1D interpolation schemes to multi-dimensions ($m$D).

As far as we recognize, tensorial Chebyshev interpolation \cite{chebfun,Lloyd2,Lloyd} best answers these questions among all state-of-the-art approaches.
However, Chebyshev interpolation is done on a (full) tensorial grid, which suffers from the \emph{curse of dimensionality} by requiring $|P_{A_{m,n}}| \in \Oc(n^m)$  interpolation nodes.
Therefore,  practical implementations are usually limited to dimensions $m \leq 3$.

% Here, we show that solving a unisolvent interpolation problem can provide higher (in fact, exponential)
% convergence rates, and yield a computationally efficient and numerically stable algorithm for computing actual instances using only sub-exponentially many nodes with space dimension.
%

Here, we therefore provide a generalized notion of \emph{unisolvent interpolation nodes}, which allows for non-tensorial grids on which we reach exponential
convergence rates for the Runge function, as a prominet example of a  Trefethen function. We show that the number of nodes required scales sub-exponentially with space dimension. We therefore believe that the present generalization of unisolvent nodes to non-tensorial grids is key to lifting the curse of dimensionality.  Our results also directly inspire an efficient algorithm to practically solve high-dimensional interpolation problems. We therefore provide a numerically robust and computationally efficient algorithm and its software implementation, and we use it to empirically verify our theoretical predictions.
Combining sub-exponential node numbers with exponential approximation rates, non-tensorial unisolvent nodes are thus able to \emph{lift the curse of dimensionality} for multivariate interpolation tasks.

In the following extensive introduction we  discuss subjects across these fields and summarize the state of the art from previous works.

\subsection{Runge's phenomenon -- Approximation theory}
Already in 1D, it is long known \cite{faber,jackson1913,runge} that for any sequence of interpolation nodes $P_n \subseteq \Omega$, $n \in \N$
there exists at least one continuous function $f : \Omega \lo \R$  that can not be approximated by interpolation on $P_n$, i.e.,
\begin{equation}\label{CW}
  Q_{f,n} \centernot{\xrightarrow[n \rightarrow \infty]{}} f \,, \qquad \text{where}\quad  Q_{f,n}(q) = f(q) \,,\,\forall\, q \in P_n\,.
\end{equation}
Instead, the approximation quality of an interpolation polynomial is sensitive to the choice of the interpolation nodes $P_n \subseteq \Omega$. In other words:
Interpolating $f$ with a polynomial $Q_{f,n}$ of increasing degree $n \in \N$ does not guarantee that the interpolant $Q_{f,n}$ converges to $f$.
This fact is famously known as \emph{Runge's phenomenon} \cite{runge}. Hence, a universal interpolation scheme
that approximates all continuous functions does not exist.

Note that while multi-dimensional versions of the Weierstrass Theorem state that any continuous function $f \in C^0(\Omega,\R)$
can be approximated by polynomials in principle \cite{glimm,Sob,stephen}, asking whether such polynomials can be determined by interpolation on specified nodes $P_A$ is in regard of Runge's phenomenon,
Eq.~\eqref{CW}, a different question.
%
% However, it does not suffice to answer Questions~\ref{Q3} theoretically and ensure {\em existence} of generic interpolation.

Whatsoever, any answer to Questions~\ref{Q2} that is to be of practical relevance
must provide a recipe to construct interpolation nodes $P_A$ that allow efficient approximation while resisting the curse of dimensionality in terms of Question~\ref{Q1}.

\subsection{Lifting the curse of dimensionality}

Recently, Lloyd N. Trefethen \cite{Lloyd2} proposed a way of delivering  a potential solution to the problem: For continuous functions $f : \Omega \lo \R$
that are analytic in the unbounded \emph{Trefethen domain} (a genralization of a Bernstein ellipse) $N_{m,\rho}\subsetneq \Omega =[-1,1]^m$, of  radius  $\rho >1$, an upper bound on the convergence rate applies:
\begin{equation}\label{Rate}
  \| f - Q_{f,A_{m,n,p}}\|_{C^0(\Omega)} \in  \li\{\begin{array}{ll}
                                                                      \Oc_\ee(\rho^{-n/\sqrt{m}})  &\,, \quad p =1 \\
                                                                      \Oc_\ee(\rho^{-n}) &\,, \quad p =2\\
                                                                      \Oc_\ee(\rho^{-n})  &\,, \quad p =\infty
                                                                      \end{array}
\re.\,,
\end{equation}
where $g \in \Oc_\ee(\rho^{-n})$ iff $g \in \Oc((\rho-\ee)^{-n})$, $\forall \ee >0$.

The multi-index sets $A_{m,n,p} =  \{\alpha \in \N^m \mi \|\alpha\|_p \leq n \} \subseteq \N^m$ generalize the notion of polynomial degree
to multi-dimensional $l_p$-degree.
This suggests that interpolating a function with respect to
the polynomial space $\Pi_{m,n,2} =\mathrm{span} \{x^\alpha\}_{\alpha \in A_{m,n,2}}$ spanned by all $l_2$-monomials {\em can reach} the same convergence rate as interpolating with respect to the $l_\infty$-degree $A_{m,n,\infty}$,
while $l_1$-degree can not reach such a fast rate. Indeed, this expectation is validated and argued to be genuine in \cite{Lloyd2}.

The number of coefficients  $|A_{m,n,1}| = \binom{m+n}{n} \in \Oc(m^n)$ required is of polynomial cardinality for $p=1$, whereas $|A_{m,n,2}|\in o(n^m)$ is of sub-exponential size for $p=2$,
but $|A_{m,n,\infty}| = (n+1)^m$ scales exponentially with dimension $m \in \N$ for $p =\infty$.
Thus, combining sub-exponential node numbers with exponential approximation rates, interpolation with respect to $l_2$-degree polynomials might yield a way of lifting the curse of dimensionality and answering Question~\ref{Q1}.

However, in \cite{Lloyd2} the interpolants $Q_{f,A_{m,n,1}},Q_{f,A_{m,n,2}}$ were computed by regression over a full Chebyshev grid $P_{A_{m,n,\infty}}$,
which requires evaluating $f$ on the exponentially many $|A_{m,n,\infty}| = (n+1)^m$ nodes $P_{A_{m,n,\infty}}$.
In other words: There is no numerically stable and efficient algorithm known that can compute $Q_{f,A_{m,n,2}}$ by evaluating $f$ only on $|A_{m,n,2}|$ unisolvent nodes while reaching the optimal Trefethen approximation rates.

\subsection{Interpolation with optimal Trefethen rates -- The core contribution of this article}
\label{sec:CD}

Here, we answer Questions~\ref{Q1}--\ref{Q2}.
To do so, we generalize the notion of unisolvent nodes $P_{A}$, $A\subseteq \N^m$ to non-tensorial grids. This allows us to extend Newton (NI) and Lagrange (LI) interpolation to arbitrary-dimensional spaces such that:
\begin{enumerate}
 \item[P1)] $P_{A}$ are given by non--tensorial, non--symmetric grids, $A = A_{m,n,p}$, scaling sub-exponentially with dimension $m$ if $p<\infty$.
 % \item[P2)] For $A=A_{m,n,p}$, $p=2$, the number $|P_{A_{m,n,2}}| \in o(n^m)$ of nodes scales sub-exponentially with space dimension.
 \item[P2)] Computing the uniquely determined multivariate interpolant $Q_{f,A} \in \Pi_A$,
 requires $\Oc(|A|^2)$ or $\Oc(|A|)$ runtime for NI and LI, respectively, and $\Oc(|A|)$ memory.
 \item[P3)] Evaluating the interpolant $Q_{f,A}(x)$ at any argument $x \in \R^m$ requires $\Oc(|A|)$ or $\Oc(|A|^2)$ runtime for NI and LI, respectively, and $\Oc(|A|)$ memory.
 % \item[P4)] We prove that for any Sobolev function $f \in H^k(\Omega,\R) \subseteq C^0(\Omega,\R)$ with $k>m/2$,
 % interpolation on $P_{A_{m,n,p}}$, $1\leq p \leq \infty$ yields an interpolant with $Q_{f,A_{m,n,p}}\xrightarrow[n \rightarrow \infty]{}  f $ uniformly on $\Omega=[-1,1]^m$.
  \item[P4)] Numerical experiments validate that the  Runge function $f(x) = 1/(1+10\|x\|^2)$ can be interpolated on $P_{A_{m,n,2}}$ to machine precision up to dimension $5$ while reaching
the optimal Trefethen rate $\Oc(\rho^{-n})$ stated in Eq.~\eqref{Rate} with the theoretically predicted Trefethen radius $\rho$.
\end{enumerate}
We want to add the following remarks:
\begin{remark}
 If one only requires the nodes $P_A$ to be unisolvent, then they do not have to be given by a (sub) grid at all. The nodes used for the present $m$D Newton interpolation are given by a sub-grid,
             but that sub-grid is neither symmetric nor tensorial. Nevertheless, we reach the optimal exponential Trefethen rates on these grids.
            If we would add nodes to make the grid symmetric or tensorial, then
            the number of nodes of the resulting (sparse) tensorial grid would scale exponentially $\Oc(n^m)$ with space dimension $m \in \N$. In contrast, our proposed interpolation nodes scale sub-exponentially $o(n^m)$ and
            thus \emph{lift the curse of dimensionality.}
 \end{remark}
 \begin{remark}
 We complement the established notion of unisolvent nodes by the \emph{dual notion of unisolvence}. That is:  For given arbitrary nodes $P$, determine the  polynomial space $\Pi$ such that
$P$ is unisolvent with respect to $\Pi$. In doing so, we revisit earlier results by Carl de Boor and Amon Ros~\cite{deBoor2,deBoor} and answer their question from our perspective.
 \end{remark}
 \begin{remark}
 Combining our multivariate Lagrange interpolation with the dual notion of unisolvence allows establishing a polynomial regression scheme for scattered data on planar and curved manifolds. In particular, we demonstrate that
the level set function $L$ of the torus $\T^2_{R,r}= L^{-1}(0)$, its gradient $\nabla L$, and the Runge function $f$ restricted to $\T^2_{R,r}$ can be approximated to high (machine) precision with interpolation nodes $P \subseteq \T^2$ sampled uniformly at random on the torus.
\end{remark}

{\bf In summary:} We answer Questions~\ref{Q1}--\ref{Q2} by establishing an efficient $m$D interpolation scheme that can approximate a generic class of functions and,
at least empirically, reaches the proposed exponential approximation rate for strongly varying Trefethen functions, such as the Runge function $f(x) = 1/(1+10\|x\|^2)$, requiring only a sub-exponential number
of non-tensorial interpolation nodes. Combining sub-exponential node numbers with exponential approximation rates, non-tensorial unisolvent nodes are thus able to \emph{lift the curse of dimensionality} for multivariate interpolation tasks.

\subsection{Related Work} The importance of the present problem is manifested in the large number of attempts that were made to solve this central problem of applied mathematics.
Consequently, an exhaustive overview of multivariate interpolation schemes can not be given here.
In the following, we restrict  ourselves to mention the most relevant state-of-the-art approaches.

\subsubsection{Interpolation in tensorial (sub) grids}
There are approaches, e.g.~\cite{dyn, Guenther, kuntz, sauertens}, that interpolate polynomials on (sparse) tensorial grids, which are know to be unisolvent for the corresponding tensorial polynomial spaces.
This, for instance, allows to theoretically derive formulas for Lagrange polynomials on these grids \cite{sauertens}.
However, this does not mean that efficient algorithms to evaluate the resulting interpolants to machine precision are known.
Furthermore, so far none of these approaches is known to reach the optimal Trefethen approximation rates when requiring the number of nodes of the underlying tensorial grids to
scale sub-exponential with space dimension. As the numerical experiments in Section~\ref{sec:NUM} suggest, we believe that only non-tensorial grids are able to lift the curse of dimensionality, which requires
non--tensorial interpolation schemes such as those presented here.

\subsubsection{Weierstrass-type Approximation}
Several approaches \cite{webster,chkifae,cohen2,cohen,cohen3,migli} are available to realize $m$D Weierstrass-type approximations by computing the $L^2$--projections onto a-priori specified
orthogonal polynomials.
However, the linear convergence rate of the Bernstein approximation \cite{bernstein1912} is reflected in the circumstance
that these approaches are prevented from approximating a generic class of functions, but are limited to well-behaving bounded analytical or holomorphic functions
occurring, for instance, as solutions of elliptic PDEs. In these scenarios, reasonable uniform approximations of the function $f$ can be
reached by sparse samples that avoid the curse of dimensionality in high dimensions $m \in \N$, $m \leq 16$.
However, when asking such approaches to deliver approximations to machine precision, or to leave the tight class of well-behaving functions,
their resistance to the curse of dimensionality disappears already for low dimensions.

\subsubsection{Multivariate splines} \label{intro:APP}
A prominent and well established alternative to the above approaches
is the multivariate spline interpolation by Carl de Boor et al. ~\cite{Boor:BS,Boor:tensorSP,Boor:wings,Boor:SP}.
The approximation result \cite{boorapp} for bivariate splines due to de Boor is given as:
\begin{theo}[Carl de Boor] Let $f : \Omega \subseteq \R^2 \lo \R$ be a \mbox{$(n+1)$}-times continuously differentiable  bivariate function, $\Delta$ be a triangulation, and
$S_{f,n,\Delta} = \li\{g \in C^\rho(\Omega,\R) \mi g_{|\delta} \in \Pi_{m,n,1}\,, \,\, \forall\, \delta \in \Delta \re\}$, $ n >3\rho +1$
be the space of \emph{piecewise polynomial functions} of degree $n$.
Then there exists $c(\Delta) >0$ such that
 \begin{equation} \label{RC}
  \mathrm{dist}(f , S_{f,n,\Delta}) \leq c(\Delta) \| D^{n+1}f\|_{C^0(\Omega)}|\Delta|^{n+1}\,,
\end{equation}
where $|\Delta| = \max_{T \in \Delta} |T|$ is the mesh size.
\end{theo}
This result states that any \emph{sufficiently smooth} function $f$ can be approximated by piecewise polynomial functions, which allows to approximate $f$ by Hermite or spline interpolation.
Generalizations of this result rely on this fact and are formulated in a similar manner~\cite{Boor:BS,Boor:tensorSP,Boor:SP}.

Despite its fundamental importance in linking interpolation and approximation, the above result and the resulting interpolation algorithms have some weak points:
\begin{enumerate}[A)]
 \item The strong regularity assumption, i.e., the \mbox{$(n+1)$}-fold differentiablity of $f$;
 \item The error bound in Eq.~\eqref{RC} only guarantees a polynomial convergence rate, but no exponential convergence;
 \item The approach is sensitive to the curse of dimensionality, i.e., the number of polynomial coefficients $C_{m,n}$ scales exponentially with dimension: $|C_{m,n}| \in \Oc(n^m)$.
\end{enumerate}
Especially for Trefethen functions, such as the Runge function $f(x)=1/(1+10\|x\|^2)$,  $(B)$ prevents spline interpolation.
Several improvements have been presented, including Floatman--Hormann interpolation~\cite{floater_rates,floater}, that reach better approximation quality than splines.
However, all of them share the above weaknesses (A,B,C), as we demonstrate in the numerical experiments of Section \ref{sec:NUM}.
Though, approximations of lower accuracy might be reached faster then by polynomial interpolation, this makes these approaches incapable for answering Question \ref{Q1} when higher-precision
approximations are required. The multivariate polynomial interpolation method presented here reaches this goal.

\subsection{Notation}

Let $m,n \in \N$, $p\geq 1$. We denote by  $e_1=(1,0,\dots,0), \dots,e_m = (0,\dots,0,1) \in \R^m$ the standard basis, by $\|\cdot\|$ the euclidean norm on $\R^m$, and by $\|M\|_p$ the $l_p$-norm
of a matrix $M\in \R^{m\times m}$. Further, $A_{m,n,p} \subseteq \N^m$ denotes all multi-indices $\alpha =(\alpha_1,\dots,\alpha_m)\in \N^m$ with $\|\alpha\|_p = (\alpha_1^p+\dots+\alpha_m^p)^{1/p}\leq n$.
We order a finite set $A\subseteq \N^m$, $m \in \N$  of multi-indices with respect to the lexicographical order $\leq_L$ on $\N^m$ starting from $x_m$ to $x_1$, e.g.
$(5,3,1)\leq_L (1,0,3) \leq_L (1,1,3)$. Then, $\alpha_{\min},\alpha_{\max}$ shall denote the minimum and maximum of $A=\{\alpha_{\min},\dots,\alpha_{\max}\}$ with respect to $\leq_L$. We call $A$
\emph{downward closed} or \emph{complete} iff there is no $\beta = (b_1,\dots,b_m) \in \N^m \setminus A$
with $b_i \leq a_i$, $ \forall \,  i=1,\dots,m$ for some $\alpha = (a_1,\dots,a_m) \in A$. In \cite{cohen3} the terminology ``downward closed set'' is used.
Note that $A_{m,n,p}$ is complete for all $m,n\in \N$, $p\geq 1$.
Given $A \subseteq \N^m$ complete and a  matrix $R_A \in \R^{|A|\times|A|}$ we slightly abuse notation by denoting
\begin{equation}
 R_A = (r_{\alpha,\beta})_{\alpha,\beta \in A} = (r_{i,j})_{1 \leq i,j \leq |A|} \,,
\end{equation}
with $\alpha,\beta$ being the $i$-th, $j$-th entry of $A$ ordered by $\leq_L$, respectively.

We consider the \emph{real polynomial ring} $\R[x_1,\dots,x_m]$  in $m$ variables and denote by $\Pi_m$  the $\R$-\emph{vector space of all real polynomials} in $m$ variables.
Further,
$\Pi_{A} \subseteq \Pi_m$ denotes the \emph{polynomial subspace} induced by $A$ and
generated by the \emph{canonical basis} given by the monomials $x^\alpha=x_1^{\alpha_1}\cdots x_m^{\alpha_m}$ with $\alpha \in A$. For $A =A_{m,n,p}$ we write $\Pi_{m,n,p}$ to mean $\Pi_{A_{m,n,p}}$.
Given a polynomial $Q(x) = \sum_{\alpha \in A} c_\alpha x^{\alpha}$, $A \subseteq \N^m$, we call $\max_{\alpha \in A\,, c_\alpha \not =0} \|\alpha\|_p$ the $l_p$-degree of $Q$. As it will turn out, the notion of
$l_p$-degree plays a crucial role for the approximation quality of polynomial interpolation.  While $A_{1,n,p} =\{0,\dots,n\}$, considering $A \subseteq \N^m$ for $m >1$ \emph{generalizes the concept of polynomial degree to multi-dimensions.}

Throughout this article $\Omega=[-1,1]^m$ denotes the $m$-dimensional \emph{standard hypercube} and $C^0(\Omega,\R)$ the
\emph{Banach space}
of continuous functions $f : \Omega \lo \R$ with norm $\|f\|_{C^0(\Omega)} = \sup_{x \in \Omega}|f(x)|$.
% Finally, we use the standard \emph{Landau symbols} $\Oc(\cdot), o(\cdot)$
% $$f \in \Oc(g)  \Longleftrightarrow \lim \sup_{x\rightarrow \infty} \frac{|f(x)|}{|g(x)|} \leq \infty \,, \qquad  f \in o(g)  \Longleftrightarrow \lim_{x\rightarrow \infty} \frac{|f(x)|}{|g(x)|} =0 \,.$$

\section{Unisolvent nodes}
\label{sec:UN}

While the pioneering works by Kuntzmann, Guenther, and Roetman~\cite{Guenther,kuntz}  and extensions in \cite{chung1977lattices} proposed a partial answer to Question~\ref{Q2}
we provide a generalized notion of unisolvent nodes for multivariate polynomial interpolation with respect
to arbitrary finite-dimensional polynomial spaces $\Pi \subseteq \Pi_m$, i.e., even in the case where $\Pi \not = \Pi_A$ is not induced by a complete or downward closed set $A\subseteq \N^m$.
We start by reviewing the usual notion of unisolvence.

\subsection{The notion of unisolvence}
We consider a (not necessarily downward closed or complete) finite set $A \subseteq \N^m$ of multi-indices, a set of interpolation nodes
$P_A= \{p_{\alpha_{\min}},\dots,p_{\alpha_{\max}}\} \subseteq \R^m$, a set $B_A=\{q_{\alpha_{\min}},\dots,q_{\alpha_{\max}}\}$ of multivariate polynomials e.g. $q_{\alpha}(x) = x^\alpha$,
the polynomial space $\Pi_A = \mathrm{span}(B_A)$ generated by the $q_{\alpha}$, and a function $f : \R^m \lo\R$. The \emph{multivariate Vandermonde matrix} is given by
\begin{equation}
 V(P) = \big (q_{\beta}(p_{\alpha})\big)_{\alpha,\beta \in A} \,.
\end{equation}
For $q_{\alpha}(x) = x^\alpha$ this results in the classic notion $ V(P) = \big (p_{\alpha}^\beta\big)_{\alpha,\beta \in A}$.
If $V(P)$ is (numerically) invertible then one can interpolate $f$ by solving the linear system of equations
\begin{equation*}
 V(P)C =F \,,  \quad C= (c_{\alpha_{\min}},\dots,c_{\alpha_{\max}}) \,, \,\, \quad F= (f(p_{\alpha_{\min}}),\dots,f(p_{\alpha_{\max}}))
\end{equation*}
using $\Oc( |A|^3)$ operations\footnote{While algorithms of lower complexity exist, such as the Strassen algorithms or the Coppersmith-Winograd algorithm, their break-even points are reached only for problems so large that memory becomes the limiting factor.}. Indeed,
\begin{equation}
 Q_{f,A}(x)=\sum_{\alpha \in A}c_\alpha q_\alpha \,\, \in \Pi_A
\end{equation}
yields the unique interpolant of $f$ on $P_A$, i.e.,  $Q_{f,A}(p)=f(p)$ for all $p \in P_A$.
We call a set of nodes $P_A$ \emph{unisolvent}  with respect to $\Pi_A$ if and only if $V(P)$ is invertible, i.e., if and only if $\ker V(P) =0$.
The condition $\ker V(P) =0$ is equivalent to requiring that there exists no hypersurface $H = Q^{-1}(0)$
generated by a polynomial $0\not =Q \in \Pi_A$ with $P \subseteq H$. Indeed, the coefficients $C$ of such a polynomial would be a non-trivial solution of $V(P)C=0$.

However, even if $P$ is unisolvent, as is well known and shown in our previous work \cite{PIP1}, the inversion of the matrix $V$ becomes numerically ill-conditioned
when represented in the canonical basis $q_{\alpha}(x) =x^\alpha$, $\alpha \in A$.
Therefore, alternative interpolation schemes with better numerical condition and lower computational complexity are desirable.
While previous approaches to addressing this problem relied on tensorial interpolation schemes \cite{dyn, Guenther, kuntz, sauertens}, we here propose a different approach.

\subsection{Construction of unisolvent nodes}
In the following,  we develop our concept of unisolvent nodes that generalizes over previous works~\cite{Guenther,kuntz,PIP1,PIP2,IEEE}.
We start by stating the definitions on which our concept rests:
\begin{definition}[Transformations] An \emph{affine transformation}  $\tau : \R^m \lo \R^m$, $m \in \N$,
is a map $\tau(x) = Bx +b$, where $B \in \R^{m\times m}$ is an invertible matrix and $b \in \R^m$. An \emph{affine translation} is an affine transformation with $B=I$ the identity matrix. A
\emph{linear transformation} is an affine transformation with $b=0$.
\end{definition}

 \begin{lemma}\label{lemma:TRF} Any affine transformation $\tau : \R^m \lo \R^m$, $m \in \N$, induces a ring isomorphism $\tau^*: \R[x_1,\dots,x_m]\lo \R[x_1,\dots,x_m]$, i.e, the induced transformation
$$ \tau^* : \Pi_m \lo \Pi_m \quad \text{given by}\quad  \tau^*(Q)(x) = Q(\tau(x)) \quad  \forall \, x \in \R^m $$
is a linear transformation, such that
\begin{enumerate}
\item[i)]    $\tau^*(\lambda Q_1 + \mu Q_2) = \lambda\tau^*(Q_1) + \mu \tau^*(Q_2)$ for all $Q_1,Q_2 \in \Pi_m$, and $\lambda,\mu \in \R$.
 \item[ii)]  $\tau^*(Q_1Q_2) = \tau^*(Q_1)\tau^*(Q_2)$ for all $Q_1,Q_2 \in \Pi_m$, and $\tau^*(1)=1$.
 \item[iii)] $\tau^*(Q_1/Q_2) = \tau^*(Q_1)/\tau^*(Q_2)$ for all $Q_1,Q_2 \in \Pi_m$, and $\tau^*(1)=1$.
\end{enumerate}
\end{lemma}

\begin{proof} While $i),ii)$ are straightforward to prove $iii)$ follows by $ii)$ using the identity
 $\tau^*(Q_1)=\tau^*(1\cdot Q_1) = \tau^*((Q_2/Q_2)Q_1) = \tau^*(Q_2)\tau^*(Q_1/Q_2)$.
\end{proof}

\begin{definition}If $\Pi \subseteq \Pi_m$, $m \in \N$, is a polynomial subspace, then we call $\tau : \R^m \lo \R^m$ a \emph{canonical transformation} with respect to $\Pi$ if
and only if $\tau$ is an affine transformation such that the induced transformation
$\tau^* : \Pi \lo \Pi_m$ maps onto $\Pi$, i.e., $\tau^*(\Pi)\subseteq \Pi$.
\end{definition}

\begin{remark}
 Note that any affine translation and any affine transformation with $B$ being a diagonal matrix
 is a canonical transformation with respect to $\Pi_{m,n,p}$, $m,n\in N$, $1 \leq p \in \R\cup\{\infty\}$.
\end{remark}

Using these definitions, we formalize the concept of unisolvent nodes as:
\begin{definition}[Unisolvent nodes] Let $m \in \N$ and $\Pi \subseteq \Pi_m$ be a polynomial subspace.
We call a finite non-empty set $\emptyset \not = P \subseteq \R^m$ \emph{unisolvent} with respect to $\Pi$ if and only if there exists no non-zero polynomial
$$Q \in \Pi\setminus\{0\} \quad  \text{with} \quad  Q(p)=0 \,,\, \forall \,\,p \in P\,.$$
Let further $H\subseteq \R^m$ be a hyperplane defined by a linear polynomial $Q_H\in \Pi_{m,1,1}\setminus \{0\}$, i.e., $H = Q_H^{-1}(0)$ such that any affine transformation
$\tau_H : \R^m \lo \R^m$ with $\tau_H(H) = \R^{m-1} \times \{0\}$ is canonical with respect to $\Pi$.
We consider
\begin{align}
 \Pi_{|H} &= \li \{ Q \in \Pi \mi \tau_H^*(Q) \in \Pi \cap (\Pi_{m-1}\times\{0\})\re\} \label{PIH} \\
 \Pi_{|H}^\perp &= \li \{ Q \in \Pi_m \mi Q_H Q \in \Pi\re\}\,. \nonumber
\end{align}
We call $P$ \emph{unisolvent with respect} to $(\Pi, H)$ if and only if
\begin{enumerate}
 \item[i)] There is no polynomial $Q \in \Pi_{|H}$ with $\tau^*_H(Q) \not = 0$ and $Q(P\cap H) =0$.
 \item[ii)] There is no polynomial $Q \in \Pi_{| H}^\perp\setminus \{0\}$ with $Q(P\setminus H) =0$.
\end{enumerate}
\end{definition}

\begin{figure}[t!]
\includegraphics[scale=0.15]{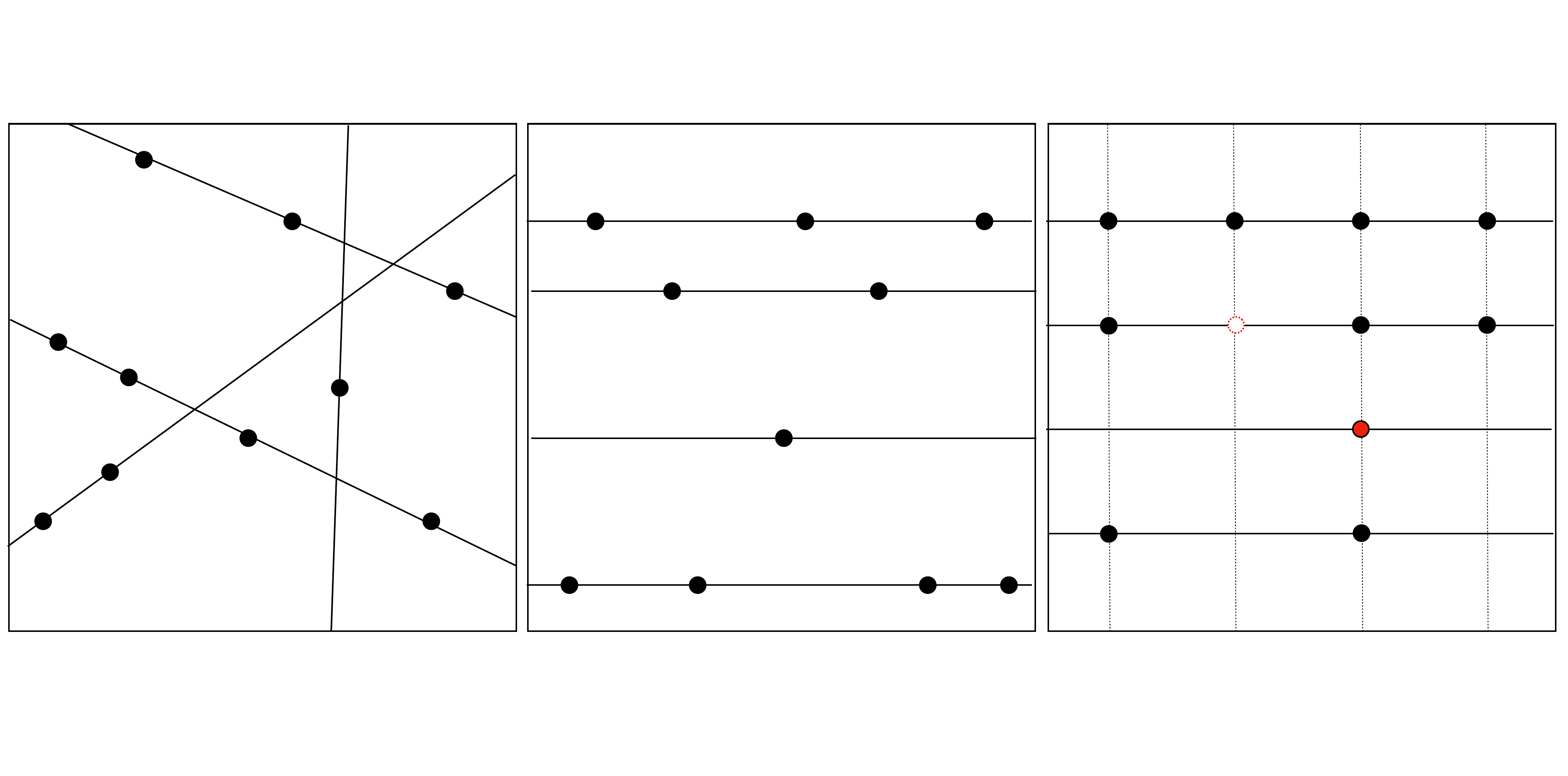}
\vspace{-0.5cm}
\caption{Examples of unisolvent nodes $P_A$ for $A= A_{2,3,1}$ in general (left) and on (irregular) grids (middle, right). Non-tensorial nodes are indicated in red with missing
symmetric counterparts shown as open symbols (right). \label{Fig:UN}}
\end{figure}

\begin{theorem} Let $m \in \N$, $ \Pi \subseteq \Pi_m$ a  polynomial subspace, $P \subseteq \R^m$ a finite set, and $H = Q_H^{-1}(0)$ be a hyperplane of co-dimension 1 defined by a polynomial $Q_H \in \Pi_{m,1,1}\setminus\{0\}$
such that:
\begin{enumerate}
 \item[i)]  The affine transformation $\tau_H : \R^m \lo \R^m$ with $\tau_H(H) = \R^{m-1} \times \{0\}$  induces a canonical transformation $\tau^*_H : \Pi \lo \Pi$.
 \item[ii)] $P$ is unisolvent with respect to $(\Pi,H)$.
\end{enumerate}
Then $P$ is unisolvent with respect to $\Pi$.
\label{theorem:UN}
\end{theorem}

\begin{proof} Let $Q \in \Pi$ with $Q(P)=0$.
We consider the affine transformation $\tau_H : \R^m \lo \R^m$ with $\tau_H(H) = \R^{m-1} \times \{0\}$ and the projection $\pi_{m-1}: \Pi_m \lo \Pi_{m-1} \times \{0\}$.
We consider
\begin{equation*}
 Q_1 = \tau^{*-1}_H\pi_{m-1}\tau^*_H(Q) \in \Pi_H\,, \quad \text{and}\quad  Q_2= (Q -Q_1)/Q_H\,.
\end{equation*}

{\bf Step 1:} We show that $Q_2 \in \Pi_{H}^\perp$.  Certainly, $Q_2$ is well defined on $\R^m\setminus H$.
Furthermore, we note that $\tau_H^*(Q_H) = \lambda x_m$, $\lambda \in \R\setminus \{0\}$. W.l.o.g. we assume $\lambda=1$ and use Lemma \ref{lemma:TRF}~$iii)$ to reformulate Eq.~\eqref{Q2} as
\begin{align*}
 Q_2 &=  \tau^{*-1}_H\big(\tau^{*}_H(Q) - \pi_{m-1}\tau^*_H(Q)\big)\big /(\tau^{*-1}_H (\tau^{*}_H(Q_H))   \\
 &= \tau^{*-1}_H\big((\tau^{*}_H(Q) - \pi_{m-1}\tau^*_H(Q))/x_m\big) .
\end{align*}
Since $Q_0:=\tau_H^*(Q)- \pi_{m-1}\tau^*_H(Q)$ is a polynomial consisting of monomials all sharing the variable $x_m$, the quotient $(\tau^{*}_H(Q) - \pi_{m-1}\tau^*_H(Q))/x_m \in \Pi$ and therefore $Q_2 \in \Pi$.
Further, by Lemma \ref{lemma:TRF}~$ii)$ we obtain
$$Q_HQ_2 = \tau_H^{*-1}(x_m)\tau_H^{*-1}(Q_0 /x_m)=\tau_H^{*-1}(Q_0) \in \Pi.$$
Hence, $Q_2 \in \Pi_{| H}^\perp$ as claimed.

{\bf Step 2:} We show that $Q=0$. Indeed, $Q(p) = Q_1(p) =0$ for all $p \in P\cap H$ implies that $Q_1=0$ in regard of assumption~$i)$.Consequently, $Q_HQ_2(p) =0$ for all $p\in P\setminus H$.
Since $Q_H(p)\not =0$ for all $p \in P\setminus H$ we get $\tau_{H}^{*-1}(Q_0)(p) =0$, $\forall p \in P\setminus H$, which, due to assumption~$ii)$ yields that $\tau_{H}^{*-1}(Q_0)=Q_2=0$. Thus, $Q=0$
is the zero polynomial and therefore $P$ is unisolvent with respect to $\Pi$.

\end{proof}
In Fig.~\ref{Fig:UN} we show examples of unisolvent nodes in 2D, generated by recursively applying Theorem~\ref{theorem:UN}. This illustrates how the notion of unisolvence presented here extends beyond notions resting on specified (sparse) grids.
%However, the case of regular (sparse) grids will be crucial later on.

%
\begin{figure}[t!]
\includegraphics[scale=0.15]{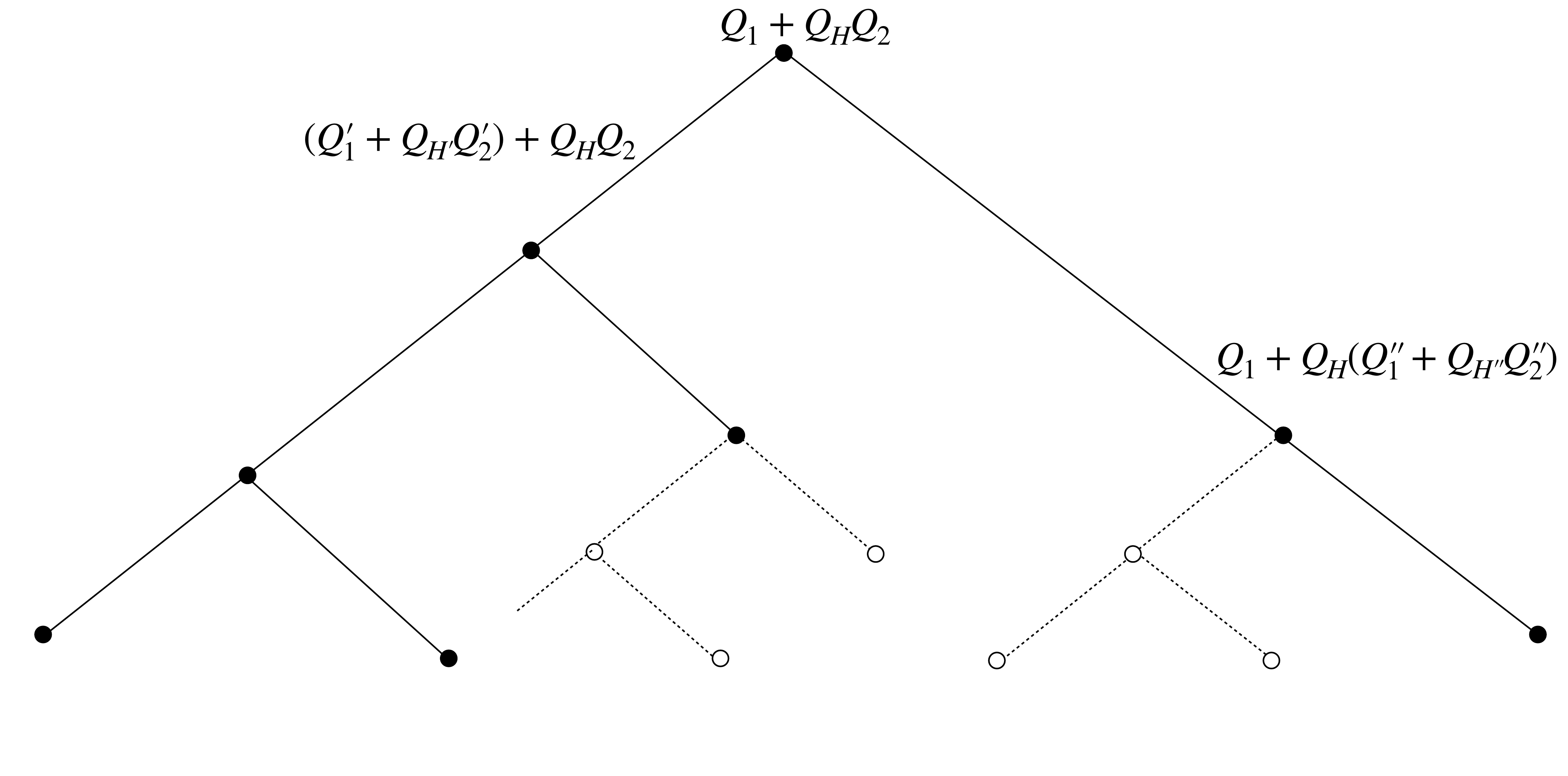}
\vspace{-0.5cm}
\caption{Recursive interpolation using the generalized divided difference scheme from Corollary~\ref{cor:GDDS}. \label{Fig:TREE}}
\end{figure}
\begin{theorem}\label{GS} Let the assumptions of Theorem \ref{theorem:UN} be fulfilled and $f : \R^m \lo \R$ be a function.
Assume that there are polynomials $Q_1 \in \Pi_{|H}, Q_2 \in \Pi_{| H}^\perp$ with $\Pi_{|H}, \Pi_{| H}^\perp$ from Eq.~\eqref{PIH}, such that:
\begin{enumerate}
 \item[i)]  $Q_1(p) = f(p)\,,\, \forall\, p \in P\cap H$
 \item[ii)] $Q_2(p) = (f(p) - Q_1(p))/Q_H(p)\,,\, \forall\, p \in P\setminus H$.
\end{enumerate}
Then $Q = Q_1 +Q_HQ_2 \in \Pi$ is the unique polynomial in $\Pi$ that interpolates $f$ on $P$, i.e., $Q(p)=f(p) \,,\, \forall\, p \in P$.
\label{theorem:PIP}
\end{theorem}
\begin{proof} Indeed, $Q_H\not = 0$ on $\R^m \setminus H$ implies that $Q(p)=f(p) \,,\, \forall\, p \in P$. Thus, $Q$ interpolates $f$ on $P$.
To show the uniqueness of $Q$ let $Q' \in \Pi$ interpolate $f$ on $P$.
Then $Q-Q' \in \Pi$ and $(Q-Q')(p) = 0 \,, \, \forall p \in P$.  Due to Theorem \ref{theorem:UN} we have that $P$ is unisolvent with respect to
$\Pi$. Thus, $Q'-Q$ has to be the zero polynomial, proving that $Q$ is uniquely determined in $\Pi$.
\end{proof}
By recursion, Theorem~\ref{GS} yields a general divided difference scheme for polynomial spaces $\Pi_A$ that are not induced by downward closed sets $A \subseteq \N^m$. The recursion is illustrated in Fig.~\ref{Fig:TREE}
and rephrases our earlier results \cite{PIP1}:

\begin{corollary}[Generalized Divided Differences] Let $\Pi \subseteq \Pi_m$, $m\in \N$, be a finite-dimensional polynomial space with $\dim \Pi =N\in \N$. Assume that there are unisolvent nodes $P$,
\begin{equation*}
 P = (P_0 \cup P_{1}) = (P_{0,0} \cup P_{0,1}) \cup (P_{1,0} \cup P_{1,1}) = \cdots = \bigcup_{\alpha \in A} P_\alpha\,, \,\, A\subseteq \N^m,
\end{equation*}
generated by recursively applying Theorem \ref{theorem:UN} with respect to hyperplanes
\begin{equation*}
H_{0,1} \subseteq \R^m, \quad H_{0,0,0,1}, H_{1,0,1,1} \subseteq \R^{m-1} \,, \dots,  H_{\alpha,\beta} \subseteq \R^2\,, \alpha, \beta \in A,
\end{equation*}
such that $P_1 \cap H_{0,1} = P_{0,1}\cap H_{0,0,0,1}= P_{1,1}\cap H_{1,0,1,1}= \dots =P_\beta \cap H_{\alpha,\beta} = \emptyset$ as illustrated in Fig.~\ref{Fig:UN} (left).
Then the unique interpolant $Q_f \in \Pi$ of any function $f : \R^m \lo \R$ can be computed in $\Oc(N^2)$.
\label{cor:GDDS}
\end{corollary}
\begin{proof} The proof follows by using Theorem \ref{GS} as an induction argument on $N$, observing that polynomials $Q \in \Pi$ can be evaluated in $\Oc(N^2)$ and classic 1D Newton interpolation requires $\Oc(n^2)$ runtime \cite{LIP}.
\end{proof}
While Corollary~\ref{cor:GDDS} proves the existence of a quadratic-runtime interpolation algorithm
even in the general case of irregular unisolvent nodes (Fig.~\ref{Fig:UN}, left), the next statement allows us
to derive a more suitable approach for implementing such an algorithm in practice:

\begin{figure}[t!]
\includegraphics[scale=0.15]{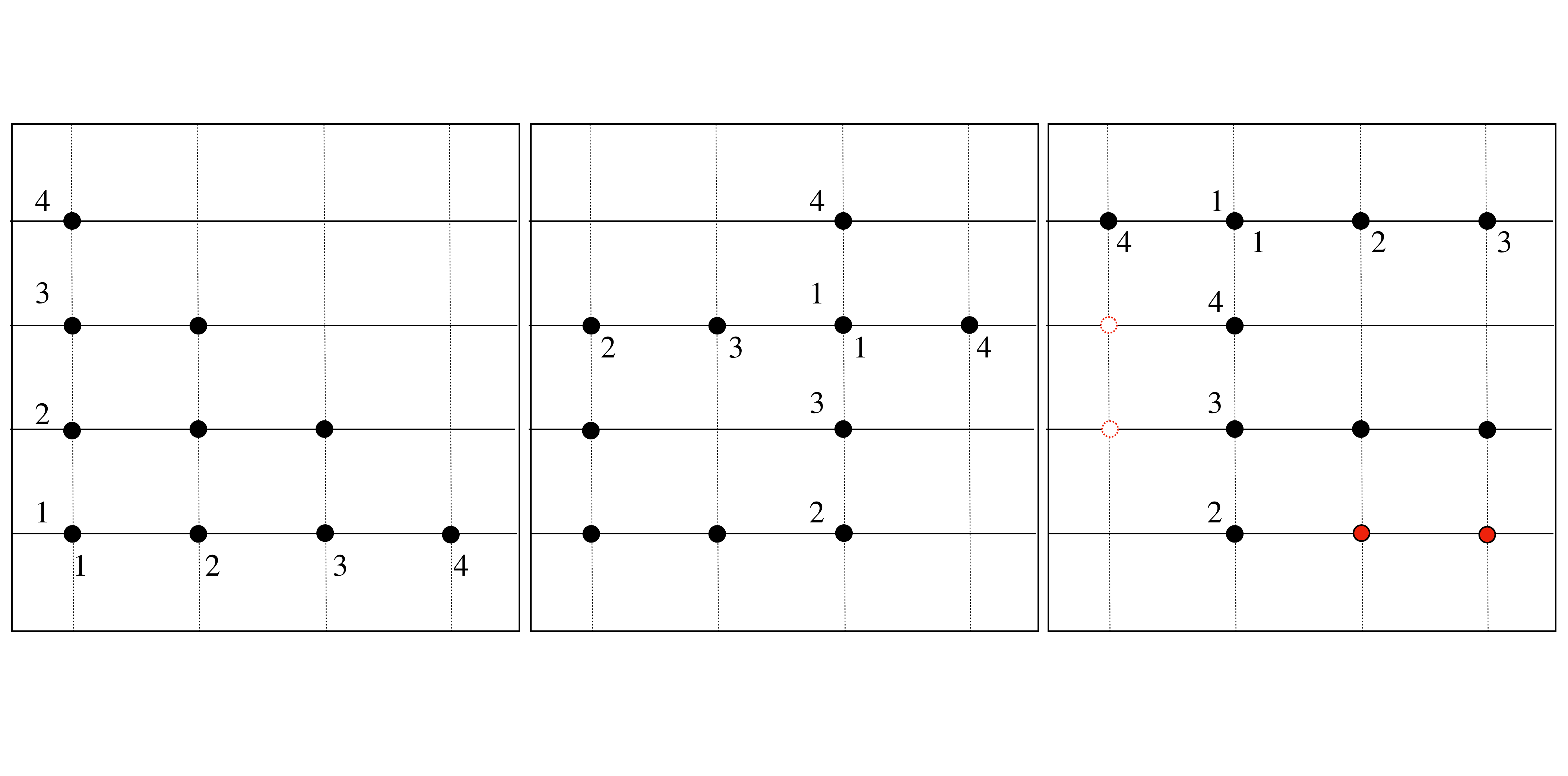}
\vspace{-0.5cm}
\caption{Examples of unisolvent nodes for $A= A_{2,3,1}$ (left, middle) and $A_{2,3,2}$ (right). Note that $(2,2) \in A_{2,3,2}\setminus A_{2,3,1}$ generates an extra node. Orderings in $x,y$--directions are indicated as well as
non-tensorial nodes in red with missing
symmetric counterparts shown as open symbols. \label{Fig:UN2}}
\end{figure}

\begin{corollary}\label{cor:grid} Let $m \in \N$, $A\subseteq \N^m$ be a complete set of multi-indices, and $\Pi_A \subseteq \Pi_m$ by the polynomial sub-space induced by $A$.
We consider the generating nodes given by the grid
\begin{equation}\label{GP}
 \mathrm{GP}= \oplus_{i=1}^m P_i\,, \quad P_i =\{p_{0,i},\dots,p_{n_i,i}\} \subseteq \R \,, \,\,\,n_i=\max_{\alpha \in A}(\alpha_i)\,,
\end{equation}
where the $P_i$ are arbitrary finite sets. Then, the node set
$$ P_A = \li\{ (p_{\alpha_1,1}\,, \dots \,, p_{\alpha_m,m} ) \mi \alpha \in A\re\} $$
is unisolvent with respect to $\Pi_A$.
\end{corollary}

\begin{proof} We argue by induction on $m$ and $|A|$. If $m=1$ then the claim follows from the fact that $\dim \Pi_A= |A|$ and no polynomial $Q\in \Pi_A$ can vanish on $|A|$ distinct nodes
$P_A=\{(p_{\alpha(1),1}) \mi \alpha \in A\}$. The claim becomes trivial for $|A|=1$. Now assume that $m >1$ and $|A|>1$. We consider $A_1=\li\{\alpha \in A \mi \alpha_m =0\re\}$, $A_2 = A \setminus A_1$.
By decreasing $m$ if necessary and w.l.o.g., we can assume that $A_2 \not = \emptyset$. Consider the hyperplane
$H =\{ (x_1,\dots,x_{m-1},p_{0,m})\mi (x_1,\dots,x_{m-1}) \in \R^{m-1}\}$ and
$Q_H \in \Pi_{m,1,1}$ with $Q_H(x)= x_m-p_{0,m}$. By induction we have that $P_A$ is unisolvent with respect to $(\Pi_A,H)$. Thus, we finish the proof by Theorem \ref{theorem:UN} and induction.
\end{proof}
\begin{definition}[Essential assumptions] We say that the \emph{essential assumptions} hold with respect to $A\subseteq \N^m$ and $P_A\subseteq \R^m$, where $m \in \N$ and $A$ is a complete (i.e., downward closed) set of multi-indices,
if and only if there exist generating nodes
\begin{equation}
 \mathrm{GP}= \oplus_{i=1}^m P_i\,, \quad P_i =\{p_{0,i},\dots,p_{n_i,i}\} \subseteq \R \,, \,\,\,n_i=\max_{\alpha \in A}(\alpha_i)\,,
\end{equation}
and the unisolvent nodes $P_A$ are given by
$$ P_A = \li\{ (p_{\alpha_1,1}\,, \dots \,, p_{\alpha_m,m} ) \mi \alpha \in A\re\} \,.$$
Unless further specified, the generating nodes $ \mathrm{GP}$ are arbitrary.
\label{EA}
\end{definition}

\begin{remark}
It is important to note that the nodes $P_A$ are not sampled from a grid, but generate a sub-grid. Consequently, even though the index sets $A$ are assumed to be downward closed,
the flexibility in ordering the $P_i$ results in unisolvent nodes $P_A$ that may induce \emph{non-tensorial or non-symmetric grids}. This can be seen in
Fig.~\ref{Fig:UN} (right), Fig.~\ref{Fig:UN2} (right), and Fig.~\ref{Nodes}, where there are nodes with $p=(p_x,p_y) \in P_A$, but $(p_y,p_x)\not \in P_A$. In Fig.~\ref{Fig:UN2} (left, middle) examples of a tensorial and symmetric grids are shown.
\end{remark}

The nodes shown in Fig.~\ref{Nodes} result in high approximation power of the interpolation, as proven in Section \ref{sec:APP}. They are generated by choosing $\mathrm{GP}= \oplus_{i=1}^m (-1)^i\Cheb_n^{0}$,
where the Chebyshev extremes $\Cheb_n^{0}$ defined in Eq.~\eqref{CHEB} are Leja ordered \cite{leja}.
Since these $P_A$ form a non-tensorial grid, previous interpolation approaches \cite{dyn,sauertens} cannot be used. We therefore establish efficient and numerically robust interpolation schemes for such non-tensorial unisolvent nodes in the next section.

\begin{figure}[t!]
\includegraphics[scale=1.05]{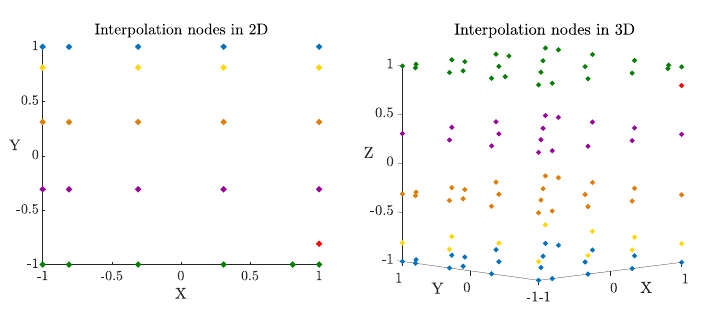}
\vspace{-0.5cm}
\caption{Unisolvent nodes $P_A$ in 2D (left) and 3D (right) with respect to $A_{m,n,p}$ for dimensions $m=2,3$, $n=5$, $p=2$, and Leja ordered \cite{leja} generating nodes $ \mathrm{GP} = \oplus_{i=1}^m(-1)^m\Cheb_n^{0}$.
Nodes belonging to the same line/plane are colored equally.} \label{Nodes}
\end{figure}

\section{Multivariate Newton interpolation}\label{sec:NEWT}

We use the above generalized concept of unisolvence on non-tensorial nodes to provide a natural extension of the classic Newton interpolation scheme to arbitrary dimensions. The extension presented here relies on recursively applying Theorem \ref{theorem:PIP} and Corollary \ref{cor:grid}. We start by defining:
\begin{definition}[Multivariate Newton polynomials] Let the essential assumptions (Definition \ref{EA}) be fulfilled with respect to $A\subseteq \N^m$ and $P_A\subseteq \R^m$. Then, we define the \emph{multivariate Newton polynomials} by
 \begin{equation}\label{Newt}
  N_\alpha(x) = \prod_{i=1}^m\prod_{j=0}^{\alpha_i-1}(x_i-p_{j,i}) \,, \quad \alpha \in A\,.
 \end{equation}
\end{definition}
Indeed, in dimension $m=1$ this reduces to the classic definition of Newton polynomials~\cite{gautschi,Stoer,Lloyd}.

\begin{definition}[Multivariate divided differences] \label{def:DDS}  Let the essential assumptions (Definition \ref{EA}) be fulfilled with respect to $A\subseteq \N^m$ and $P_A\subseteq \R^m$. Further let $f : \R^m \lo \R$ be a function.
Then, we recursively define the {\em multivariate divided differences}:
\begin{alignat*}{5}
p_{\alpha,0,j} & = p_\alpha  \,,  \quad  &  p_{\alpha,i,j} &= p_{\beta}\,,\quad  &  \beta_j =i \,, \beta_k = \alpha_k \,,  \, \forall k \not = j\,,     \\
F_{\alpha,0,m} & =f(p_{\alpha}) \,, \quad   & F_{\alpha,0,j} & =F_{\alpha,\alpha_{j+1},j+1}\,, \quad  & \text{for}\,\,\,   1 \leq j < m\,,
\end{alignat*}
and
$$
F_{\alpha,i,j} =\frac{F_{\alpha,i-1,j}(p_{\alpha}) - F_{\alpha,i-1,j}(p_{\alpha,i-1,j})}{(p_{\alpha}-p_{\alpha,i-1,j}) }  \,, \,\,\, \text{for}\,\,\, i \leq \alpha _j \,.
$$
We call  $F_{\alpha,0,0}= F_{\alpha,\alpha_1,1}$  the \emph{Newton coefficients} of $Q_{f,A} \in \Pi_A$.
\label{MVDD}
\end{definition}
In dimension $m=1$, this definition recovers the classic \emph{divided difference scheme} of \emph{1D Newton Interpolation}~\cite{gautschi,Stoer,Lloyd}.

Using these definitions we state the main result of this section, generalizing Newton interpolation to $m$D:
\begin{theorem} \label{theorem:DDS} Let the essential assumptions (Definition \ref{EA}) be fulfilled with respect to $A\subseteq \N^m$ and $P_A\subseteq \R^m$, and let $f : \R^m \lo \R$ be a function.
Then, the unique polynomial $Q_{f,A} \in \Pi_A$ interpolating $f$ on $P_A$, i.e., $Q_f(p) = f(p) \,, \, \forall \, p \in P_A$, can be determined in $\Oc(|A|^2)$ operations requiring $\Oc(|A|)$ storage and is given by
\begin{equation}\label{Newton}
  Q_{f,A}(x) = \sum_{\alpha \in A} c_\alpha N_{\alpha} (x)\,,
\end{equation}
where $c_\alpha =F_{\alpha,0,0}$ are the \emph{Newton coefficients} of $Q_{f,A} \in \Pi_A$.
\end{theorem}
Earlier versions of this statement were limited to the case where $P_A$ is given by a (sparse) tensorial grid \cite{dyn}.
In contrast, the present Theorem \ref{theorem:DDS} also holds for our generalized notion of non-tensorial unisolvent nodes.
In addition, Eq.~\ref{Newton} enables
a paractical implementation of a recursive algorithm. We prove Theorem~\ref{theorem:DDS}:
\begin{proof} We argue by induction on $|A|$. If $|A|=1$ then the claim follows immediately. For  $|A|>1$ we consider $A_1=\li\{\alpha \in A \mi \alpha_m =0\re\}$, $A_2 = A \setminus A_1$. By decreasing $m$ if necessary, and w.l.o.g., we can assume that $A_2 \not = \emptyset$. Let
$H =\{(x_1,\dots,x_{m-1},p_{0,m})\mi (x_1,\dots,x_{m-1}) \in \R^{m-1}\}$ and $\tau_H : \R^m\lo\R^m$ with $\tau_H(x) =x_m- p_{0,m}$ such that $\tau_H(H) = \R^m \times\{0\}$.
By assumption, we have that $\tau$ is a canonical transformation with respect to $\Pi_A$. Let $\pi_{m-1} : \R^m \lo \R^{m-1}$, $\pi_{m-1}(x_1,\dots,x_m)=(x_1,\dots,x_{m-1})$ be the natural projection.

{\bf Step 1:} We reduce the interpolation to $H$. We set  $P_1 = \pi_{m-1}\big(\tau_H(P_A\cap H)\big)$ and consider $f_0 : \R^m \lo \R$ with
\begin{equation}\label{ftau}
 f_0(x_1,\dots,x_{m-1}) = f\big(\tau_H^{-1}(x_1,\dots,x_{m-1},0)\big)=f\big(x_1,\dots,x_{m-1},p_{0,m})\,.
\end{equation}
Let $M_{\alpha}(x)$, $\alpha \in A_1$, be the Newton polynomials with respect to $A_1$, $P_1$. Then induction yields that the coefficients $d_\alpha \in \R$ of the unique polynomial
$$ Q_{f_0,A_1}(x_1,\dots,x_{m-1}) = \sum_{\alpha \in A_1 } d_\alpha M_{\alpha}(x_1,\dots,x_{m-1})$$
interpolating $f_0$ on $P_1$ can be determined in less than $D_0|A_1|^2$ operations, $D_0 \in \R^+$, while requiring a linear amount of storage.
Consider the  natural embedding $i_{m-1}^*:\Pi_{m-1} \hookrightarrow \Pi_m$. Then
\begin{align*}
 \bar M_\alpha(x_1,\dots,x_m)= i_{m-1}^*(M_{\alpha})(x_1,\dots,x_{m-1}) &= M_{\alpha}(x_1,\dots,x_{m-1}) \,\,\, \text{and}\,\,\,  \\
Q_1(x_1,\dots,x_m) = i_{m-1}^*(Q_{f_0,A_1})(x_1,\dots,x_{m-1}) &= Q_{f_0,A_1}(x_1,\dots,x_{m-1})
\end{align*}
yields
$\bar M_\alpha(x) = N_{\alpha}(x)$, for all $\alpha \in A_1$, $x \in \R^m$. Further,  $Q_1 \in \Pi_A$ is given by
$$
Q_1(x_1,\dots,x_m)= \sum_{\alpha\in A_1}d_{\alpha}N_{\alpha}(x_1,\dots,x_m)
$$
and satisfies $Q_1(p) =f(p)$ for all $p \in P_A \cap H$.

{\bf Step 2:} We interpolate on $\R^m \setminus H$.  Observe that $Q_1$ is constant in direction $x_m$, i.e., $Q_1(x_1,\dots,x_{m-1},x_m) =Q_{f_0,A_1}(x_1,\dots,x_{m-1})$. Thus, by Eq.~\eqref{ftau},
$$Q_1(q_1,\dots,q_{m-1},q_m) =f(q_1,\dots,q_{m-1},p_{0,m})\,\, \text{for all}\,\, (q_1,\dots,q_{m-1},q_m) \in P_A\,.  $$
In light of this fact, and for $f_1(x) = (f(x) - Q_1(x))/Q_H(x)$, it requires $D_1|A_2|$, $D_1 \in \R^+$, operations to compute
\begin{equation}\label{F}
 F_{\alpha,1,m}=\frac{f(p_{\alpha,1,m}) - f(p_{\alpha,0,m})}{p_{\alpha}-p_{\alpha,0,m}}=\frac{f(p_\alpha) - Q_1(p_\alpha)}{Q_H(p_\alpha)}=f_1(p_\alpha)
\end{equation}
 for all $p_{\alpha} \in P_2=P_A\setminus H$, $\alpha \in A_2$.
Denote by $K_{\alpha}(x)$ the Newton polynomial with respect to $\Pi_{A_2},P_2$ then
induction yields that the coefficients
$e_\alpha \in \R$, $\alpha \in A_2$, of the unique polynomial
$$Q_{f_1,A_2}(x_1,\dots,x_m)= \sum_{\alpha \in A_2}e_\alpha K_{\alpha}(x_1,\dots,x_m)\,.$$
interpolating $f_1$ on $P_2$ can be determined in less than $D_2|A_2|^2$, $D_2 \in \R^+$, operations while requiring linear storage. Due to Eq.~\eqref{Newt} we observe that  $Q_H(x)K_\alpha(x) = N_{\alpha}(x)$ for all $\alpha \in A_2$. By Corollary \ref{cor:grid} we have that
$P_A$ is unisolvent and therefore Theorem \ref{theorem:PIP} implies that
the unique polynomial $Q\in \Pi_A$ interpolating $f$ on $P_A$ is given by:
\begin{equation}\label{split}
  Q_{f,A}(x)=Q_1(x) + Q_H(x)Q_2(x)=\sum_{\alpha \in A_1}d_\alpha N_{\alpha}(x) + \sum_{\alpha \in A_2}e_\alpha N_{\alpha}(x)\,.
\end{equation}
Following Definition \ref{def:DDS} and using Eq.~\eqref{F}, one readily observes that $d_\alpha = c_\alpha$, $\,\forall\,\alpha \in A_1$, and that $e_\alpha=c_\alpha$, $\,\forall\,\alpha \in A_2$.
Thus, we have proven Eq.~\eqref{Newton}. In total, the computation can be done in less than $D_0|A_1|^2 + D_1|A_2|+D_2|A_2|^2\leq \max\{D_0,D_1/2,D_2\}(|A_1|+|A_2|)^2 \in \Oc(|A|^2)$ operations and $\Oc(|A_1| + |A_2|))=\Oc(|A|)$ amount of storage.

\end{proof}
As an immediate consequence, we deduce that even for non-tensorial grids $P_A$ the Newton polynomials are a basis of $\Pi_A$.
\begin{corollary}[Newton basis] \label{cor:Nbasis} Let the essential assumptions (Definition \ref{EA}) be fulfilled with respect to $A\subseteq \N^m$ and $P_A\subseteq \R^m$. Then the Newton polynomials
$$ \{N_{\alpha}\}_{\alpha \in A} \subseteq \Pi_A$$
are a basis of $\Pi_A$.
\end{corollary}
\begin{proof} Due Theorem \ref{theorem:DDS} every polynomial $Q \in \Pi_A$ can be uniquely expanded as $Q=\sum_{\alpha \in A}c_{\alpha}N_{\alpha}$, proofing the statement.
\end{proof}

\begin{corollary}[Evaluation in Newton form] \label{cor:eval} Let the essential assumptions (Theorem \ref{EA}) be fulfilled with respect to $A\subseteq \N^m$ and $P_A\subseteq \R^m$.
Further let $Q(x) = \sum_{\alpha}c_\alpha N_{\alpha}$,
$c_\alpha \in \R$, be a polynomial in Newton form. Then, it requires $\Oc(|A|)$ operations to evaluate $Q$ at some $x_0 \in \R^m$.
\end{corollary}
\begin{proof} By following the proof of Theorem \ref{theorem:DDS} and using an induction argument on the number of coefficients, Eq.~\eqref{split} yields that $Q_1,Q_2$
can be evaluated in linear time. Since the evaluation of $Q_H$ requires constant time, the claim follows.

\end{proof}

\begin{remark}
Recursively applying the splitting $Q=Q_1+Q_HQ_2$ recovers the classic Aitken-Neville evaluation algorithm \cite{muhlbach,muhlbachg,neville} for dimension $m=1$.
Evaluation is efficiently done by applying a multivariate version of the Horner scheme, see for instance~\cite{gautschi,Stoer,Jannik}.
\end{remark}

\section{Multivariate Lagrange interpolation} \label{sec:LAG}

Even though the present unisolvent nodes do not have to be tensorial, and therefore the results of \cite{sauertens,sauerL}
do not apply, we can still generalize the concept of Lagrange interpolation to multi-dimensions. For this, we define:
\begin{definition}[Lagrange polynomials] \label{def:LagP}
 Let  $m \in \N$, $A \subseteq \N^m$ be a complete set of multi-indices, and $P_A = \{p_{\alpha}\}_{\alpha \in A}$ be an unisolvent set of nodes with respect to a polynomial subspace $\Pi_{P_A} \subseteq \Pi_{m}$.
Then, we define the \emph{multivariate Lagrange polynomials}
\begin{equation}\label{DL}
  L_{\alpha} \in \Pi_{P_A}\ \quad \text{with}\quad L_{\alpha}(p_\beta)= \delta_{\alpha,\beta}\, , \,\,\, \alpha,\beta \in A\,,
\end{equation}
where $\delta_{\cdot,\cdot}$ is the Kronecker delta.
\end{definition}
For $A= A_{m,n,\infty}$ and the grid $P_A$ becomes tensorial and the above definition recovers the known notion of \emph{tensorial $m$D Lagrange interpolation}~\cite{Gasca,sauertens}, where
\begin{equation}\label{LagTens}
  L_{\alpha}(x)= \prod_{i=1}^m l_{\alpha_i} (x) \,, \quad l_{\alpha_i} (x) = \prod_{j=0,j \not =\alpha_i}^n \frac{x_i-p_{j,i}}{p_{\alpha_i,i} - p_{j,i}}  \,, \quad  \alpha \in A.
\end{equation}
The following theorem then generalizes the classic facts known for
1D Lagrange interpolation~\cite{berrut,Lloyd} and for tensorial $m$D Lagrange interpolation~\cite{berrut,Lloyd} to the non-tensorial multi-dimensional case:
\begin{corollary}[Lagrange basis]\label{cor:Lag} Let the assumptions of Definition \ref{def:LagP} be fulfilled. Then:
\begin{enumerate}
 \item[i)]  The Lagrange polynomials $L_\alpha \in \Pi_A$ are a basis of $\Pi_A$.
 \item[ii)] The polynomial $Q_{f,A}(x) = \sum_{\alpha \in A}f(p_{\alpha})L_{\alpha}(x) \in \Pi_{A}$ is the unique polynomial interpolating $f$ on $P_A$ and can be determined in $\Oc(|A|)$ operations.
\end{enumerate}
\end{corollary}

\begin{proof} To show $i)$, observe that there are $|A|$ Lagrange polynomials. Due to Corollary \ref{cor:Nbasis}
we deduce $\dim \Pi_A = |A|$. Given $c_{\alpha} \in \R$, $\alpha \in A$ such that
$\sum_{\alpha \in A} c_\alpha L_\alpha = 0$, the unisolvence of $P_A$ implies that the polynomial $Q(x)= \sum_{\alpha \in A} c_\alpha L_\alpha$ vanishes on $P_A$ and, therefore, has to be the zero polynomial.
Hence, $c_\alpha =0$ for all $\alpha \in A$,
implying that the $L_\alpha \in \Pi_A$ are linear independent and thus yield a basis of $\Pi_A$. The claimed uniqueness in $ii)$ follows from $i)$, and the remaining statement holds for trivial reasons.
\end{proof}

\begin{remark}\label{rem:scale} The complexity of the interpolation depends on the choice of $A$. As introduced, the cases $A_{m,n,1}$, $A_{m,n,2}$, and $A_{m,n, \infty}$ are of special interest. While
$$ |A_{m,n,1}|= \binom{m+n}{n} = \binom{m+n}{m} \in \Oc(m^n)\cap\Oc(n^m) \,, \quad |A_{m,n,\infty}| = (n+1)^m \,,$$
explicit such formulas for $|A_{m,n,p}|$, $1<p<\infty$, are unknown. An approximation for $p=2$ is given in \cite{Lloyd2} as
$$|A_{m,n,2}| \approx (n+1)^m\mathrm{vol}(B^m_{l_2}) = (n+1)^m \frac{(\pi/ 4)^{m/2}}{(m/2)!} \approx \frac{(n+1)^m }{\sqrt{\pi m}} \li(\frac{\pi \mathrm{e}}{2m}\re)^{m/2}  \,,$$
where $B_{l_2}^m$ denotes the $l_2$-ball in dimension $m$.
Thus, for fixed degree $n \in \N$ Lagrange interpolation is of polynomial complexity $\Oc(m^n)$ for $p=1$, of sub-exponential complexity $o(n^m)$ for $p=2$, and of exponential complexity $\Oc(n^m)$ for $p=\infty$.
\end{remark}

\section{Regression on scattered data} \label{sec:SCAT}

We address the question of how to interpolate a function $f \in C^0(\Omega,\R)$ if the interpolation nodes $P$ can not be chosen according to Definition~\ref{EA},
but are arbitrarily given and fixed. We start by noting that if the locations of the nodes are distributed uniformly at random, then the nodes are unisolvent with probability $1$~\cite{PIP2}.
Based on this observation, we state:

\begin{corollary}[Regression on given nodes] \label{cor:Scat}Let the essential assumptions (Definition~\ref{EA})
be fulfilled with respect to $A\subseteq \N^m$ and $P_A\subseteq \R^m$, let
 $f : \R^m \lo \R$ be a function, and let $\overline P\supseteq \overline P_A$ be any set of nodes containing at least one unisolvent node set
$\overline P_A=\{\bar p_{\alpha}\}_{\alpha \in A}\subseteq \R^m$ with respect to $\Pi_A$.
Denote by $\overline F =(f(\bar p ))_{\bar p \in P} \in \R^{|P|}$.
Then the following is true:
\begin{enumerate}
 \item[i)] The Lagrange coefficients $C_{\mathrm{Lag}}=(c_{\alpha_{\min}},\dots, c_{\alpha_{\max}})$ of $Q_{f,A}$ are uniquely determined by solving
\begin{equation*}
 R_A C_{\mathrm{Lag}} \approx  \overline F
  \,, \,\, \text{where}\,\,\,  R_A=(r_{i,\alpha}) \in \R^{|\overline P|\times|A|}  \quad \text{with}\quad r_{i,\alpha}= L_{\alpha}(\bar p_{i})\,,\bar  p_i \in \overline P
\end{equation*}
using standard least-squares regression: $C_{\mathrm{Lag}} = \mathrm{argmin}_{x \in \R^{|A|}}\|R_A x -\overline F\|_2^2$.
\item[ii)] If $f \in \Pi_A$ is a polynomial, then $Q_{f,A}= f$.
\end{enumerate}
\end{corollary}
\begin{proof} Since $\overline P_A$ is assumed to be unisolvent, the matrix $R_A$ is of full rank.
Thus, $i)$ follows from the existence of a unique solution for the convex least-squares optimization problem. Then $ii)$ follows by the unisolvence and by the fact that $Q_{f,A}(p) = f(p)$ for all $p \in \overline P_A$.
\end{proof}

We validate the practical feasibility of this regression method in Section \ref{sec:NUM2}. In particular, we empirically observe the matrix $R_A$ to be well conditioned even for complex regression problems,
reflecting the proper choice of unisolvent nodes.

\section{The dual notion of unisolvence}
\label{sec:Dual}
Rather than constructing unisolvent nodes {\em de novo}, Carl de Boor and Amon Ros~\cite{deBoor2,deBoor} considered the dual problem of unisolvence:
for a given polynomial space $\Pi$ and a set of nodes $P \subseteq \R^m$ that is not unisolvent with respect to $\Pi$,
find a maximum subset $P_0 \subseteq P$ and a polynomial subspace $\Pi_{P_0} \subseteq \Pi$,
such that $P_0$ is unisolvent with respect to
$\Pi_{P_0}$~\cite{deBoor2,deBoor}.
We suggest to rephrase the problem in more abstract mathematical terms by considering the map
\begin{equation}\label{Gamma}
  \Gamma_k : \Pc_k \lo \GR(k,X)\,,     \quad X= \Pi_{m,k,\infty} ,
\end{equation}
where $\Pc_k = \li\{ P \subseteq \R^m \mi |P| =k\re\}$ is the set of all finite subsets of $\R^m$ of cardinality $k$, and $\GR(k,X)$ is the \emph{Grassmann manifold}, i.e.,
the smooth manifold that consists of all $k$-dimensional subspaces of
the vector space $X$ \cite{milnor}. In particular, $\GR(1,\R^m) = \mathbb{RP}^{m-1}$ and $\GR(1,\C^m) = \mathbb{CP}^{m-1}$ are the real and complex \emph{projective spaces}, respectively~\cite{dieudonne1971elements,hatcher_vec}.

We start by stating:
\begin{theorem} \label{theorem:Gamma} Let $m,k \in \N$, $\Pc_k = \li\{ P \subseteq \R^m \mi |P| =k\re\}$ be the set of all subsets of $\R^m$ with  cardinality $k$, and $X = \Pi_{m,k,\infty}$ be the space of all polynomials
with $l_\infty$-degree at most $k$. Then, there is one and only one polynomial subspace $\Pi_P \subseteq X$ such that $P$ is unisolvent with respect $\Pi_P$.
In particular, the map
\begin{equation*}
  \Gamma_k : \Pc_k \lo \GR(k,X)\,,     \quad \Gamma_k(P) = \Pi_P \, ,
\end{equation*}
with $\GR(k,X)$ denoting the \emph{Grassmann manifold} \cite{milnor}, is well defined and smooth.
\end{theorem}
An extended and etailed version of this stament is meanwhile given in our contribution \cite{GPLS}.

\begin{remark}
Note that for $m=1$ we have $\Pi_{m,k,\infty} = \Pi_{1,k,1}$ and $\GR(k,X) = X=\Pi_{1,n,1}$. Since  \mbox{$n+1$} nodes are unisolvent in dimension 1, Theorem \ref{theorem:Gamma} becomes trivial, and $\Gamma_k(P)\equiv X$ is constant in that case.
\end{remark}

\begin{proof} Let $P_{A_{m,k,\infty}}$ be a set of unisolvent nodes with respect to $X =\Pi_{m,k,\infty}$, generated according to the essential assumptions in Definition~\ref{EA}.
Denote by $L_\alpha \in X$, $\alpha \in A_{m,k\infty}$ the corresponding Lagrange polynomials. Fix an ordering $P =\{p_0,\dots,p_k\}$ and consider the matrix $R=(r_{i,\alpha}) \in \R^{k \times K}$,
$K =|A_{m,k,\infty}| = (k+1)^m$, defined by
$$r_{i,\alpha} = L_\alpha (p_i)\,, \quad i=1,\dots,k, \,\,\, \alpha_{\min} \leq_L \alpha\leq_L \alpha_{\max}\,,\,\, \alpha \in A_{m,k,\infty}\,.$$
Let $\mu =\rank(R)$ be the rank of $R$ and $D = \mathrm{diag}(\overbrace{1,\dots,1}^{\mu},\overbrace{0,\dots,0}^{k-\mu}) \in \R^{k \times k}$ the diagonal matrix with the first $\mu$ entries equal to 1 and all others equal to 0.
Let further $C\in \R^{K \times k}$ be a solution of $RC = D$ and $b_i =(c_{\alpha_{\min},i},\dots,c_{\alpha_{\max},i})$ be the rows of $C$. Then
$$ B_i(x) = \sum_{\alpha \in A_{m,k,\infty}} c_{\alpha,i}L_{\alpha}(x) \in X\,, \quad i =1 ,\dots,k$$
is the maximal set of linearly independent polynomials with $B_i(p_j) = \delta_{i,j}$, $1 \leq j \leq \mu$, where $\delta_{\cdot,\cdot}$ denotes the Kronecker delta.
The $B_i$ are the interpolants of the delta-distribution-functions
$d_{i} : \R^m \lo \R$, $d_i(x) = \delta_{x,p_i}$. Since  $P_{A_{m,k,\infty}}$ is unisolvent, the $B_i$ are uniquely determined and linearly independent.
Hence, $\mu =k$, and the $B_i$ are a basis of the uniquely determined polynomial subspace $\Pi_P = \mathrm{span} (B_i)_{i=1,\dots,k}$ for which
$P$ becomes unisolvent. Since $\GR(k,X)$ is a smooth manifold \cite{milnor}, and the $B_i$ depend smoothly on $P$, this shows that
$\Gamma_k$ is a well-defined smooth map.

\end{proof}

\begin{remark}
We revisit the work of Carl de Boor and Amon Ros~\cite{deBoor2,deBoor} by considering the subspace
$$ \widetilde Y(P) = \Gamma_k(P) \cap  \Pi \,.$$
The set $P_0 \subseteq P$ and a basis of $\widetilde Y(P)$ can be derived using \emph{Gaussian elimination} \cite{Lloyd_Num} on the associated Vandermonde matrix $V(P)$.
We note that properties of the map $P \mapsto \widetilde Y(P)$ discussed in~\cite{deBoor2,deBoor}, such as \emph{continuity} and \emph{differentiablity},
can be deduced by considering $\Gamma_k$ from Eq.~\eqref{Gamma}.
\end{remark}

However, the fact that $\dim(X) = (k+1)^m$ implies that Theorem \ref{theorem:Gamma} is of mostly theoretical interest.
Only if  $\Gamma_k(P) \approx Y$ is known to be located near a relatively low-dimensional subspace $Y \subseteq X$, e.g.~$Y = \Pi_{m,n,2}$ with small $m,n$, then Theorem \ref{theorem:Gamma}
might be of practical relevance, as the following consequence states:
\begin{theorem}\label{theorem:Dual} Let the essential assumptions (Definition~\ref{EA}) be fulfilled with respect to $A\subseteq \N^m$ and $P_A\subseteq \R^m$. Let further $P \subseteq \R^m$ be any set of nodes
that is not unisolvent with respect to $\Pi_A$,
$\Gamma_k$ be as in Theorem \ref{theorem:Gamma}, and $f : \R^m \lo \R$ be a function. Then:
\begin{enumerate}
 \item[i)] There is  a set $P_0 \subseteq P$ of maximal cardinality $k=|P_0|$ that can be determined in $\Oc(|A|^3)$ operations such that $\Gamma_k(P_0) \subseteq \Pi_A$.
 \item[ii)]  A basis $(\rho_1,\dots,\rho_{k})$ of $\Gamma_k(P_0)$ and a basis $(\mu_1,\dots,\mu_{|A|-k})$ of the quotient space $\Pi_{A}/\Gamma_k(P_0)$ can be computed in $\Oc(|A|^3)$ operations.
 \item[iii)] The Lagrange coefficients $c_\alpha \in \R$, $\alpha \in A$, $Q_{P_0,f} = \sum_{\alpha\in A} c_{\alpha}L_{\alpha}(x) \in \Gamma(P_0)$ of an interpolant of  $f$ on $P_0$
 can be computed in $\Oc(|A|^3)$ operations.
 \item[iv)] The interpolant  $Q_{P_0,f}$ is uniquely determined up to adding a polynomial $q \in \Pi_A$ with $q(p) = 0$ for all
 $p \in P$, i.e., $Q_{P_0,f} +q$ is a valid interpolant.
  \item[v)] If $\Gamma(P_0) \not = \Pi_A$, then the Lagrange coefficients $d_\alpha \in \R$ of  a polynomial $0 \not = Q_0 = \sum_{\alpha\in A} d_{\alpha}L_{\alpha}(x) \in \Pi_A$ satisfying $Q_0(p) = 0$ for all
 $p \in P$ can be computed in $\Oc(|A|^3)$ operations.
\end{enumerate}
\end{theorem}
\begin{proof} All statements follow from the following observation: We order the nodes  $P = \{p_1, \dots , p_{l}\}$, $l =|P|$. Let
$L_\alpha$, $\alpha \in A$, be the Lagrange polynomials with respect to $A,P_A$. Consider the matrix
$R_A =(r_{i,\alpha}) \in \R^{l \times |A|}$ given by
$$ r_{i,\alpha} = L_\alpha (p_i)\,, \quad \alpha _0 \leq \alpha \leq \alpha_{\max}\,, \, 1 \leq i \leq l\,.$$
By using \emph{Gaussian elimination with full pivoting} (GEFP) \cite{Lloyd_Num}, we can find an $LU$-decomposition of $R_A$. That is, there are \emph{permutation matrices} $W_1,W_2 \in \R^{l\times l}$,
a \emph{unitary lower triangular matrix} $L\in \R^{l\times |A|}$, and an \emph{upper triangular matrix}  $U \in \R^{|A|\times |A|}$ such that
\begin{equation}\label{RLU}  W_1 R_AW_2 =L U\,,
\quad \text{with}\quad  U= \li(\begin{array}{cc}
                                U_1 & U_2 \\
                                0 & 0
                               \end{array}\re)  \in \R^{l\times l}\,, U_1\in \R^{k \times k},U_2\in \R^{k \times |A|-k}
\end{equation}
and the $k$ diagonal entries  of $U_1$ do not vanish.
Consequently,  $\rank(R_A) = \rank(U) =\rank(U_1) =k \in \N$.  Let
$P_0 = \{\bar p_1,\dots,\bar p_k\}$ be the first $k$ nodes of the node set $P$ when reorderd accordingly to $W_1$ and denote with $L_{\beta_j} \in \Pi_A$, $1\leq j\leq |A|$ the Lagrange polynomials w.r.t. $P_A$
when reorderd accordingly to $W_2$. Denote with $S_A \in \R^{k\times |A|}$ the matrix given by the first $k$ rows of $R_A$ then
the Lagrange polynomials $\bar L_i$, $i =1 ,\dots,k$, with $\bar L_i(\bar p_j) = \delta_{ij}$  for all $\bar p_j \in P_0$, $ i,j=1,\dots,k$, are uniquely determined by
$$ \bar L_i(x)= \sum_{j =0}^{|A|} c_{ij}L_{\beta_j}(x)\,, \quad  C_i=(c_{ij})_{1\leq j\leq|A|} \in \R^{|A|}\,, \quad S_A C_i = e_i\,,$$
where $e_i$ is the $i$-th standard basis vector of $\R^k$. Since $\rank(R_A)=k$, the set $P_0$ is a maximal subset of $P$ with that property. While
all matrices above can be determined by GEFP in $\Oc(|A|^3)$ operations, this shows $i)$.

Due to $i)$, setting $\rho_1 = \bar L_1,\dots,\rho_k = \bar L_k \in \Pi_A$ yields a  basis of $\Gamma_k(P_0)$. Conversely,
computing $D_h\in \R^{k}$, $1 \leq h \leq |A|-k$ with
$$ U_1 D_h = - U_2 e_h$$
where $e_h$ is the $h$-th standard basis vector of $\R^{|A|-k}$. Setting $\tilde D_h = (D_h,0), \tilde e_h = (0,e_h) \in \R^{|A|}$ to be the natural embeddings of $D_h,e_h$ respectively and
and setting $d_h = \tilde D_h-\tilde e_{h}\in \R^{|A|}$ yields polynomials
\begin{equation} \label{kernel}
  \mu_h(x) = \sum_{j=1}^{|A|} d_{h,j}L_{\beta_j}(x)\,,
\end{equation}
that satisfy $\mu_h(p) = 0 $ for all $p \in P$. Hence, the $\{\mu_h\}$ yield a  basis of $\{Q \in \Pi_A \mi Q(P) =0\} \cong \Pi_A / \Gamma_k(P_0)$. Thus, we have proven $ii)$.
Finally, $iii)$-- $v)$ are direct consequences of $i)$ and $ii)$.
\end{proof}

One of the practical applications  of Theorem \ref{theorem:Dual} can be stated as follows:
\begin{corollary} \label{cor:Torus} Let the essential assumptions (Definition~\ref{EA}) be fulfilled with respect to $A\subseteq \N^m$ and $P_A\subseteq \R^m$.
Let further $Q_M \in \Pi_A$ be a polynomial and $M = Q_M^{-1}(0)$ be the affine algebraic variety
given  by the zero-level set of $Q_M$. Denote by
\begin{equation}\label{PiM}
 \Pi_M = \{Q_{| M} \mi Q \in \Pi_A \} \subseteq \Pi_m
\end{equation}
 the polynomial subspace of all restrictions $Q_{|M}$  of polynomials $Q \in \Pi_A$ to $M$.
\begin{enumerate}
 \item [i)] Assume that there is a finite set $P_0 \subseteq M$ such that $\Gamma_k(P_0) \subseteq \Pi_A$, with $k =|P_0| = |A|-1$.
 Then  $\Gamma_k(P_0) \cong \Pi_M$.
 \item[ii)] If $P_0$
  is as in $i)$, then by replacing $\Gamma_k(P_0)$ with $\Pi_M$ the statements $ii)$--$v)$ of Theorem \ref{theorem:Dual} apply.
In particular, $Q_M$ can be determined up to a constant factor.
\end{enumerate}
\end{corollary}
\begin{proof} Observe that $\Pi_M$ can be understood as the quotient of $\Pi_A$ by all polynomials $m_M:=\{Q \in \Pi_A \mi Q_{|M}\equiv 0\}$ vanishing on $M$, i.e.,
$$\Pi_M \cong \Pi_A /m_M, \quad \dim (\Pi_M) = \dim (\Pi_A) -\dim (m_M) = l \in \N\,.$$
Since $Q_M \in \Pi_A$ we have that $Q_M \in m_M$ and thereby $\dim(\Pi_M) \leq |A|-1$. On the other hand, for $P_0$ as in $i)$ we have $\Gamma_k(P_0)\subseteq \Pi_M$ with $\dim \Gamma_k(P_0) = |A|-1$. Thus,
$i)$ follows. Statement $ii)$ is then obvious. Indeed, $m_M =\mathrm{span}(Q_M) \cong \Pi_A/\Gamma_k(P_0) = \mathrm{span}(\mu)$ yields that $Q_M = c\mu$, $c \in \R\setminus\{0\}$ with $\mu$ as in Theorem \ref{theorem:Dual} $ii)$.
\end{proof}

\begin{remark}\label{PROP1} Since unisolvence is a generic property \cite{PIP2,smale}, and due to Theorem \ref{theorem:Dual}$i)$, randomly sampling $|A|$ nodes on $M$ yields a set $P_0$ as required in $i)$ with probability $1$.
Consequently, any restriction $Q_{|M} \in \Pi_M$ of a polynomial $Q \in \Pi_A$ to $M$ can be interpolated as in Theorem \ref{theorem:Dual}$iii)$.
We demonstrate the numerical stability of such an approach in Section \ref{sec:NUM2}. But before that, we discuss the approximation power of polynomial interpolation in $m$D.
\end{remark}

\section{Approximation theory} \label{sec:APP}

We address the fundamental question of how well polynomial interpolation can approximate continuous functions in $m$D.
We derive several statements that enable control over the approximation error
$$\|f - Q_{f}\|_{C^0(\Omega)}\,.$$
We
denote by $\partial^\alpha f(x) = \partial_{x_1}^{\alpha_1}\dots\partial_{x_m}^{\alpha_m}f(x)$
the partial derivative of $f$ with respect to the multi-index $\alpha$, evaluated at point $x \in \Omega$. Further, we denote by $C^0(\Omega,\R)$ the $\R$-vector space of continuous functions on $\Omega$ with norm
$\|f\|_{C^0(\Omega)}= \sup_{x \in \Omega}|f(x)|$.
For a set of multi-indices $A \subseteq \N^m$, $m \in \N$, we consider
\begin{align*}
  C_A(\Omega,\R) &= \li\{f \in C^0(\Omega,\R) \mi \partial^\alpha f \in C^0(\Omega,\R)\,, \, \forall \, \alpha \in A\re\}\,,  \\
  \|f\|_{C_A(\Omega)} &= \sum_{\alpha \in A}\|\partial^\alpha f\|_{C^0(\Omega)}.
\end{align*}
For $A=A_{m,n,1}$ the normed vector space $(C_{A}(\Omega,\R), \|\cdot\|_{C_A(\Omega)})$ coincides with the classic definition of the \emph{Banach space} $(C^n(\Omega,\R), \|\cdot\|_{C^n(\Omega)})$~\cite{Sob}.
In light of this fact, one can easily deduce that $(C_{A}(\Omega,\R), \|\cdot\|_{C_A(\Omega)})$ is a Banach space for all downward closed sets $A \subseteq \N^m$.

The central challenge we address here is to free interpolation from \emph{Runge's phenomenon}, as discussed in Section \ref{intro:APP}. Therefore, we introduce the \emph{Chebyshev nodes of first kind} and the
\emph{Chebyshev extreme nodes}
\begin{align}
\Cheb_n^{1\mathrm{st}} &= \li\{ \cos\Big(\frac{2k-1}{2(n+1)}\pi\Big) \mi 1 \leq k \leq n+1\re\} , \nonumber\\
% \Cheb_n^{0} &= \li\{ \cos\Big(\frac{k\pi}{n+2}\Big) \mi 1 \leq k \leq n\re\}  \nonumber \\
\Cheb_n^{0} &= \li\{ \cos\Big(\frac{k\pi}{n}\Big) \mi 0 \leq k \leq n\re\}\,. \label{CHEB}
\end{align}
The nodes $\Cheb_n^{1\mathrm{st}}$ are minimizers of the product $M_{P_n}(x)=\prod_{p \in P_n}|x-p|$, $|P_n|=n+1$, i.e.,
\begin{equation}\label{MIN}
 \min_{P_n \subseteq \Omega} \|M_{P_n}\|_{C^0(\Omega)}  = \frac{1}{2^n} \quad \text{for}\,\,\,P_n =\Cheb^{1\mathrm{st}}_n\,.
\end{equation}
The nodes $\Cheb_{n}^{0}$ are the extrema of $M_{\Cheb^{1\mathrm{st}}_{n+1}}$, with values oscillating between $M_{\Cheb^{1\mathrm{st}}_{n+1}}(q) \in \{-1,1\}$ for $q \in \Cheb_{n}^{0}$ \cite{gautschi,Lloyd}.

\subsection{Combining numerical accuracy with approximation ability}
The following technical detail is of central importance for the remainder of this section:
\begin{definition}[2nd essential assumption] We say that the 2nd essential assumption is fulfilled if and only if in addition to the essential assumption from Definition \ref{EA} the
generating nodes
\begin{equation}
 \mathrm{GP}= \oplus_{i=1}^m P_i\,, \quad P_i =\{p_{0,i},\dots,p_{n_i,i}\} \subseteq \R \,, \,\,\,n_i=\max_{\alpha \in A}(\alpha_i)\,,
\end{equation}
are given by  \emph{Leja-ordered}~\cite{leja} Chebyshev extreme nodes, i.e.,
$$P_i= \{p_0,\dots,p_n\} = \pm\Cheb_n^{1\mathrm{st}}, \pm \Cheb_n^{0}$$ and the following holds:
\begin{equation}\label{LEJA}
 |p_0| = \max_{p \in P}|p|\,, \quad \prod_{i=0}^{j-1}|p_j-p_i| = \max_{j\leq k\leq m} \prod_{i=0}^{j-1}|p_k-p_i|\,,\quad 1 \leq j \leq n\,.
\end{equation}
\label{EA2}
\end{definition}
Leja ordering is known to minimize numerical rounding errors in 1D Newton interpolation~\cite{breuss}. In particular, $\{p_0,p_1\} = \{-1,1\}$ holds for $\Cheb_n^{0}$, $n\geq 1$.
Consequently, as we demonstrate in Section \ref{sec:NUM}, this allows approximating highly varying functions, such as the Runge function, to machine precision.

\begin{definition}[Lebesgue constant] \label{def:LEB} Let the assumptions of Definition \ref{def:LagP} be fulfilled and $f \in C^0(\Omega,\R)$ with $Q_{f,A}(x) = \sum_{\alpha \in A}f(p_\alpha)L_{\alpha}(x)$
the interpolant of $f$ in Lagrange form.
Then we define the \emph{Lebesgue constant} as the operator norm of the interpolation operator $I_{P_A} : C^0(\Omega,\R) \lo C^0(\Omega,\R)$, $ f \mapsto Q_{f,A}$ with respect to  $P_A$, i.e.,
\begin{align*}
\Lambda(P_A) := \|I_{P_A}\|= \sup_{f\in C^0(\Omega,R)\,, \|f\|_{C^0(\Omega)}\leq 1} \|Q_{f,A}\|_{C^0(\Omega)}
             =  \Big\|\sum_{\alpha \in A} |L_{\alpha}|\Big\|_{C^0(\Omega)}
\,.
\end{align*}
\end{definition}

Since $P_A$ is unisolvent, $I_{P_A}$ is linear  and  exact on $\Pi_A$, i.e, $I_A(Q)=Q$ for all $Q \in \Pi_A$.
The Lebesgue constant provides a relative measure of the approximation quality of $P_A$ in the following sense: Let $Q^*_f \in P_A$ be an optimal polynomial approximation of $f$, then:
\begin{align}
 \|f - Q_{f,A} \|_{C^0(\Omega)} & \leq  \|f- Q^*_f \|_{C^0(\Omega)} + \|Q_{f,A} - Q^*_f\|_{C^0(\Omega)}  \nonumber\\
                                & \leq  \|f- Q^*_f \|_{C^0(\Omega)} + \|I_{P_A}(f - Q^*_f)\|_{C^0(\Omega)} \nonumber \\
                                &\leq  (1 + \Lambda(P_A))\|f- Q^*_f \|_{C^0(\Omega)} . \label{estL}
\end{align}
In dimension $m=1$, it is known that for any arbitrary sequence of interpolation nodes $P_n \subseteq \Omega$, $\Lambda(P_n)$ is unbounded. However, by choosing $P_n = \Cheb_n^{1\mathrm{st}}$ or $P_n = \Cheb_n^{0}$, the
following estimate applies~\cite{brutman,brutman2,ehlich,L0,L1,mccabe,rivlin,rivlin2}:
\begin{equation}\label{LEB}
 \Lambda(P_n)=\frac{2}{\pi}\big(\log(n) + \gamma +  \log(8/\pi)\big) + \Oc(1/n^2)\,,
\end{equation}
where $\gamma \approx  0.5772$ is the Euler-Mascheroni constant.The estimates  of \cite{cohen2}[Eq.~(3.19)] show that for $A = A_{m,n,p}$, $p=1,2$ and fixed dimension $m \in \N$, we observe algebraic growth
\begin{equation}
  \Lambda(P_{A}) = \Oc(|A|\Lambda(P_n)^m))\,,\quad |A| =  \Oc(n^m)\,.
\end{equation}
Due to Eq.~\eqref{estL}, this implies that
all functions $f$ for which the optimal approximation error $\|f- Q^*_{n,f} \|_{C^0(\Omega)} $ decreases faster than $\Oc(\Lambda(P_A)))$ can be approximated by polynomial interpolation in $P_A$.

\subsection{Approximation rate in multiple dimensions}
In practice, the question of how fast the interpolant $Q_{f,A}$ converges to $f$ is of certain interest.
Lloyd N. Trefethen recently used the famous result of Bernstein's prize-winning memoir from 1914 \cite{bernstein1912} to derive upper bounds on the
convergence rates~\cite{Lloyd2}. He assumed the function $f$ to be analytical
over generalized versions of \emph{Hooke} and \emph{Newton ellipses} \cite{arnold}. Numerical experiments suggested that
these rates are also lower bounds, with first steps already done to mathematically prove this expectation~\cite{converse}. Here we revisit these results and adapt them to our problem.

\begin{definition}
Let
$E_{h^2}^2$ be the \emph{Newton ellipse} with foci $0$ and $1$ and leftmost point $-h^2$. For $m \in \N$ and $h \in [0,1]$ we call $\rho = h + \sqrt{1 + h^2}$ the Trefethen radius
and the open region
$$ N_{m,\rho} = \li\{ (x_1,\dots,x_m) \in \C^m \mi (x_1^2 + \cdots  + x_m^2)/m \in E_{h^2}^2 \re\} $$
the \emph{Trefethen domain} \cite{Lloyd2}.
\end{definition}

We restate Lloyd N. Trefethen's statement \cite{Lloyd2} here in adapted form to match our notation. We call a continuous function  $f : \Omega \subseteq \R^m \lo\R$ analytical in the Trefethen domain
$N_{m,\rho} \subseteq \C^m$ if and only if
$f$ can be continuously extended to a function $\widetilde f : \widetilde \Omega \subseteq \C^m \lo \R$ that possesses an absolutely convergent Taylor series in $N_{m,\rho}$.

\begin{theorem}[Lloyd N. Trefethen] \label{theorem:TREF} Let the 2nd essential assumption be fulfilled with respect to $A_{m,n,p}$.
Further, assume that $f \in C^0(\Omega,\R)$ is analytical in the Trefethen domain $N_{m,\rho}$. Then
$$\|f - Q_{f,A_{m,n,p}} \|_{C^0(\Omega)} \in  \li\{\begin{array}{ll}
                                                                      \Oc_\ee(\rho^{-n/\sqrt{m}})  &\,, \quad p =1, \\
                                                                      \Oc_\ee(\rho^{-n}) &\,, \quad p =2,\\
                                                                      \Oc_\ee(\rho^{-n})  &\,, \quad p =\infty\, .
                                                                      \end{array}
\re. $$
\end{theorem}
Therefore, $g(n) \in \Oc_\ee(\rho^{-n})$ if and only if  $g(n) \in \Oc((\rho-\ee)^{-n})$ for all $\ee >0$. Thus, the above are upper bounds on the approximation rates.

\begin{remark} \label{rem:eff} The above statement indicates that the $l_2$-degree suffices for the approximation to be as good as when using the entire $P_{A_{m,n,\infty}}$ grid.
In light of Remark \ref{rem:scale}, the interpolation can therefore reach
any approximation accuracy with sub-exponential complexity $\Oc(|A_{m,n,2}|)$, $|A_{m,n,2}|\in o(n^m)$, in any dimension $m$
whenever the optimal rate for $p=2$ applies. This can be observed in the numerical experiments presented in the next section.
\end{remark}

We complete this section by extending classic approximation error estimates to arbitrary dimensions.

\subsection{Error bound in multiple dimensions}

We generalize the classic 1D approximation error bound for polynomial interpolation to arbitrary dimensions $m \in N$. For a given complete set of multi-indices $A \subseteq \N^m$,  we denote by
\begin{align*}
 \underline\partial A &= \li\{ \alpha \in A \mi \alpha + e_i \not \in A  \,\,\, \text{for all}\,\,\,  1\leq i\leq m\re\}\\
  \overline \partial A &= \li\{ \beta\in \N^m \mi \beta = \alpha + e_i \,\,\,\text{for some}\,\,\, \alpha \in \underline \partial A\,, 1\leq i\leq m\re\}
\end{align*}
the \emph{discrete inner and outer boundaries} of $A$ and define $\bar A = A \cup \overline \partial A$ to be the closure of $A$ in this sense.
If $m=1$ then for any $n \in \N$, $A = \{0,\dots,n\}$, $\underline\partial A=\{n\}$, $\overline \partial A =\{n+1\}$, and $\bar A = \{0,\dots,n+1\}$.
With this we can state:
\begin{theorem}[Approximation error] Let the essential assumptions (Definition~\ref{EA}) be fulfilled with respect to $A\subseteq \N^m$ and $P_A\subseteq \R^m$. Further, let $f \in C_{\bar A}(\Omega,\R)$.
Then, for any $ x \in \Omega$, there are $\xi_{x, \beta} \in \Omega$, $\beta \in \overline \partial A$, such that:
\begin{equation}\label{est}
 f( x)- Q_{f,A}( x)  =\sum_{\beta \in \overline \partial A}\frac{\partial^\beta f(\xi_{x,\beta} )}{\beta !} N_{\beta}(x) \,, \quad \beta! := \prod_{i=1}^m \beta_i!\,.
\end{equation}
\end{theorem}

\begin{proof} We argue by induction  on $|A|$. For $|A|=1$ we have $A = \{0\}$ and $Q_{f,A}(x)= f(p_{0})$ for all $x \in \Omega$. Thus, by the Mean Value Theorem, we have
$$ f(x) - Q_{f,A}(x)= f(x) - f(p_0) = \sum_{i=1}^m \partial_{x_i} f(\xi_{x,i})(x_i-p_{0,i})$$
for some $\xi_{x,i}\in \Omega$, yielding Eq.~\eqref{est} in this case. For $|A|>1$, we consider the subsets $A_1=\li\{\alpha \in A \mi \alpha_m =0\re\}$, $A_2 = A \setminus A_1$, and the hyperplane $H =\{(x_1,\dots,x_{m-1},p_{0,m})\mi (x_1,\dots,x_{m-1})\in \R^{m-1}\}$ defined by the polynomial $Q_{H}=(x_m-p_{0,m}) \in \Pi_{m,1,1}$.
As in Eq.~\eqref{split}, we use the splitting $Q_{f,A}(x)=Q_1(x) +Q_H(x)Q_2(x)$, $x \in \R^m$.
We denote by $x_H=(x_1,\dots,x_{m-1},p_{0,m})$ the projection of $x=(x_1,\dots,x_m)\in \R^m$ onto $H$
and recall that Theorem \ref{theorem:PIP} guarantees $Q_1(x) =Q_1(x_H)$. By decreasing $m$ if necessary and w.l.o.g., we can assume that $A_2 \not = \emptyset$. We denote by $K_{\beta}(x)$ the multivariate Newton polynomials with respect to $A_2$ and recursively compute
\begin{align}
f(x)- Q(x)  &= f(x_H) - Q_1(x_H) + Q_H(x)\left(\frac{f(x) -f(x_H)}{x_m - p_{0,m}}-Q_2(x)\right) \label{E1}\\
            &= \sum_{\beta \in \overline \partial A_1}\frac{\partial^\beta f(\xi_{x_H,\beta})}{\beta !} N_{\beta}(x_H)   + \Big(\partial_{x_m} f(\eta_x) -Q_2(x)\Big)Q_H(x) \nonumber \\
            &= \sum_{\beta \in \overline \partial A_1}\frac{\partial^\beta f(\xi_{x_H,\beta})}{\beta !} N_{\beta}(x)   +
                \sum_{\beta \in \overline \partial A_2}\frac{\partial^\beta\partial_{x_m} f(\xi_{x,\beta})}{\beta !}Q_H(x)K_{\beta}(x) \nonumber \\
            &= \sum_{\beta \in \overline \partial A}\frac{\partial^\beta f(\xi_{x,\beta} )}{\beta !} N_{\beta}(x)\,, \label{E2}
\end{align}
where we used the Mean Value Theorem for the second term in Eq.~\eqref{E1} and the fact that $Q_H(x)K_{\beta}(x)= N_{\beta(x)}$ for all $\beta \in \overline \partial A_2$  to yield Eq.~\eqref{E2}.
\end{proof}

\begin{remark}
Note that for $m=1$, Eq.~\eqref{est} reduces to the classic 1D result
$$  f(x)- Q_{f,A}(x)  =\frac{ f^{(n+1)}(\xi_{x})}{(n+1)!} \prod_{i=0}^n (x- p_i)  $$
with $P_A=\{p_0,\dots,p_n\}$, $|A|=n+1$.
This yields the known approximation error bound in 1D~\cite{gautschi}:
\begin{equation}\label{CE}
  |f(x)- Q_{f,A}(x)|  \leq \frac{ |f^{(n+1)}(\xi_{x})|}{2^n(n+1)!}  \leq \frac{ \|f^{(n+1)}\|_{C^0(\Omega)}}{2^n(n+1)!}\,, \,\,\,P_A =\Cheb_n^{1\mathrm{st}}\,.
\end{equation}
\end{remark}

Our result in Eq.~\eqref{est} provides a similar bound on the approximation error in $m$D whenever the $k$-th derivatives of $f$ are known or bounded.
However, usually these bounds are unknown. By validating the proposed Trefethen approximation rates in the next section, we even though provide a potential
control of the approximation error being applicable in practice.

\section{Numerical experiments}\label{sec:NUM}
We implement a prototype of our multivariate interpolation solver, named \emph{MIP}, in MATLAB. The code implements the multivariate divided difference scheme of Definition \ref{def:DDS}
for interpolation nodes $P_A$, $A= A_{m,n,p}$, generated by Leja-ordered Chebyshev extreme nodes, i.e.,
$\mathrm{GP} = \oplus_{i=1}^m\Cheb_n^{0}$
according to the 2nd essential assumption in Definition~\ref{EA2}.
We compare our solver with the following alternative methods:
\begin{enumerate}
 \item \emph{Chebfun} from the corresponding MATLAB package \cite{chebfun};
 \item \emph{Cubic splines} and \emph{$5^{th}$-order splines} from the MATLAB \emph{Curve Fitting Toolbox};
 \item \emph{Floater-Hormann} interpolation \cite{floater} from the R package \emph{chebpol} \cite{gaure};
 \item \emph{Multi-linear} (piecewise linear) interpolation  from  \emph{chebpol} \cite{gaure};
  \item \emph{Chebyshev} interpolation of $1^\text{st}$ kind from  \emph{chebpol} \cite{gaure};
  \item \emph{Uniform} (grid) interpolation by Chebyshev polynomials from  \emph{chebpol} \cite{gaure};
 \item \emph{Vandermonde} interpolation on $P_{A_{m,n,p}}$ in MATLAB.
\end{enumerate}
Note that apart from \emph{MIP} and \emph{Vandermonde}, all other schemes use regular grids $P_{A_{m,n\infty}}$ as interpolation nodes. Therefore, \emph{Chebfun} and \emph{Chebyshev} only deliver $l_\infty$-degree interpolations.

All implementations were benchmarked using MATLAB version R2019b, Chebfun package version 5.7.0, and  R versions 3.2.3/Linux and 3.6.2/macOS with chebpol package version 2.1.2 on a standard personal
computer (Intel(R) Xeon(R) CPU E5-2660 v3 @2.60GHz, 128GB RAM).

The code and all benchmark datasets are freely available from:
\emph{https://git.mpi-cbg.de/mosaic/polyapprox}. The implementation of \emph{MIP} is provided as a prototype, which can be used to reproduce the results presented here.
Currenlty, we optimize and re-implement the code from a software engineering perspective into an open source Python package including further algorithmic improvements,
such as the \emph{multivariate barycentric Lagrange interpolation} \cite{berrut} discussed in Section \ref{sec:Bary}.

\subsection{Approximation on the hypercube}\label{AHC}

In the first set of experiments, we illustrate the statements of Theorem˜\ref{theorem:TREF}. Therefore,
we consider the Runge function
$$f_{R}(x) =\frac{1}{1+10\|x\|^2}\,.$$
Observe that $f_R$ has poles in $z=\pm i/ \sqrt{10} \in \C$.
In light of this fact, one can show that $f_R$ is analytic in the Trefethen domain $N_{m,\rho}$ with $h = 1/\sqrt{10} \approx 0.316$ and with Trefethen radius $\rho = h + \sqrt{1 +h^2} \approx 1.365$, thus fulfilling the requirements of Theorem~\ref{theorem:TREF}.

\begin{experiment} We measure the approximation errors of the interpolants computed by the mentioned methods.
To do so, we sample 100 randomly nodes $P\subseteq \Omega$, $|P|=100$,  independently generated for each degree, but identical for all methods  and  determine $\max_{q\in P}|f(q) -Q_{f}(q)| \approx \| f -Q_{f} \|_{C^0(\Omega)}$ .
\end{experiment}

Figure \ref{Rate_2D} shows the results of this experiment in dimension $m=2$. We observe that
\emph{Chebyshev},
\emph{Chebfun}, and \emph{MIP} are the only methods that converge down to machine precision (32-bit double-precision arithmetics). The convergence rate is as stated in
Theorem \ref{theorem:TREF} and reproduces earlier results by Lloyd N.~Trefethen~\cite{Lloyd2} as introduced in Section \ref{intro}.
However, we only use the $P_{A}$, $A=A_{m,n,p}$, $p=1,2$, unisolvent nodes to determine the interpolants, whereas Trefethen computed the rates for the $l_1$- and $l_2$-degree
approximations by regression over the whole $l_\infty$-grid.
This detail might be the reason for the slight advantage of \emph{MIP} over \emph{Chebfun} and \emph{Chebyshev}
for high degrees.
Further, we recognize that the \emph{Vandermonde approach} is inaccurate and even becomes numerically unstable (rising errors) for higher degrees. It is therefore inappropriate for
approximating strongly varying functions, such as the Runge function. As expected, (Chebyshev) polynomial interpolation on uniform grids (\emph{uniform}) and \emph{multi-linear} interpolation also do not converge.

Finally, we observe that \emph{Floater-Hormann interpolation} performs better than multivariate cubic splines. It is comparable to \emph{$5^{th}$-order splines},
but reaches an accuracy of $10^{-7}$ faster than any other approach.

\begin{figure}[t!]
\includegraphics[scale=0.25]{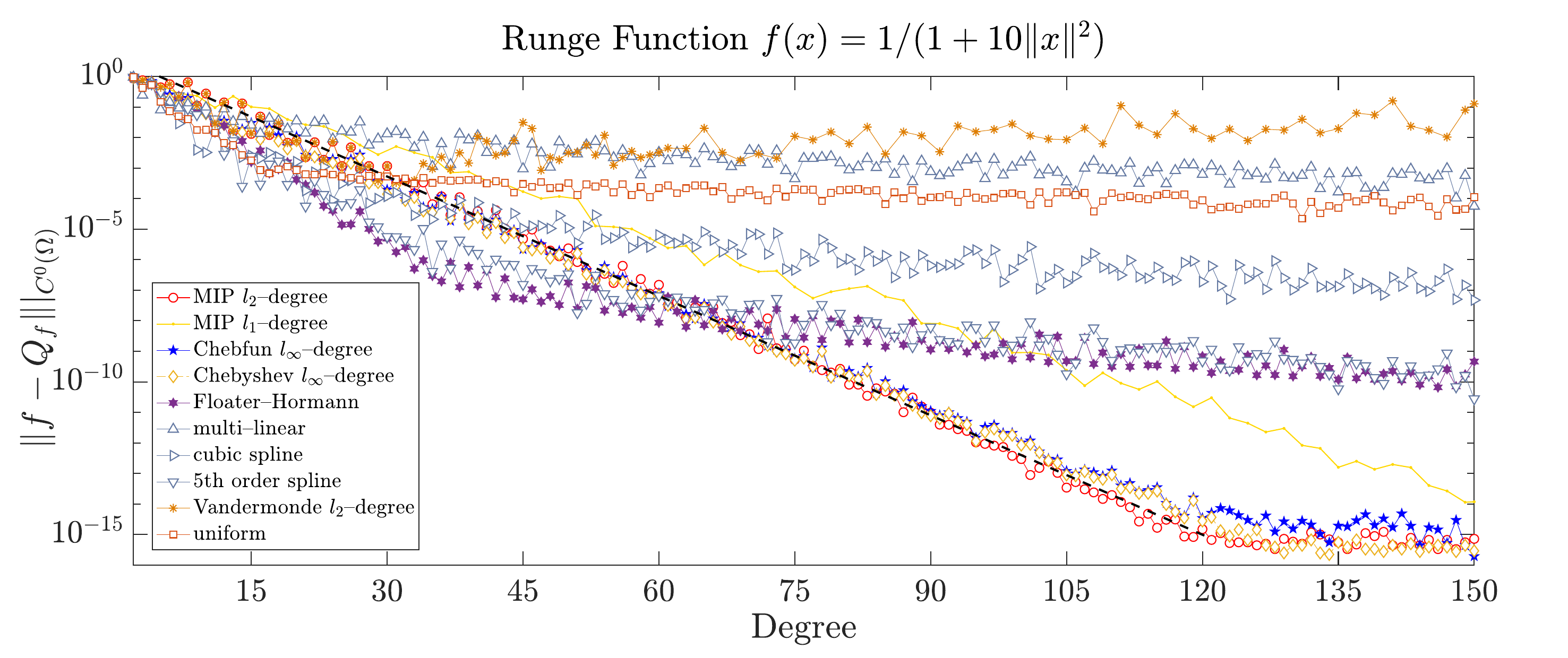}
\vspace{-0.5cm}
\caption{Approximation errors for the benchmarked methods interpolating the Runge function in dimension $m=2$. } \label{Rate_2D}
\end{figure}
\begin{figure}[t!]
 \includegraphics[scale=0.25]{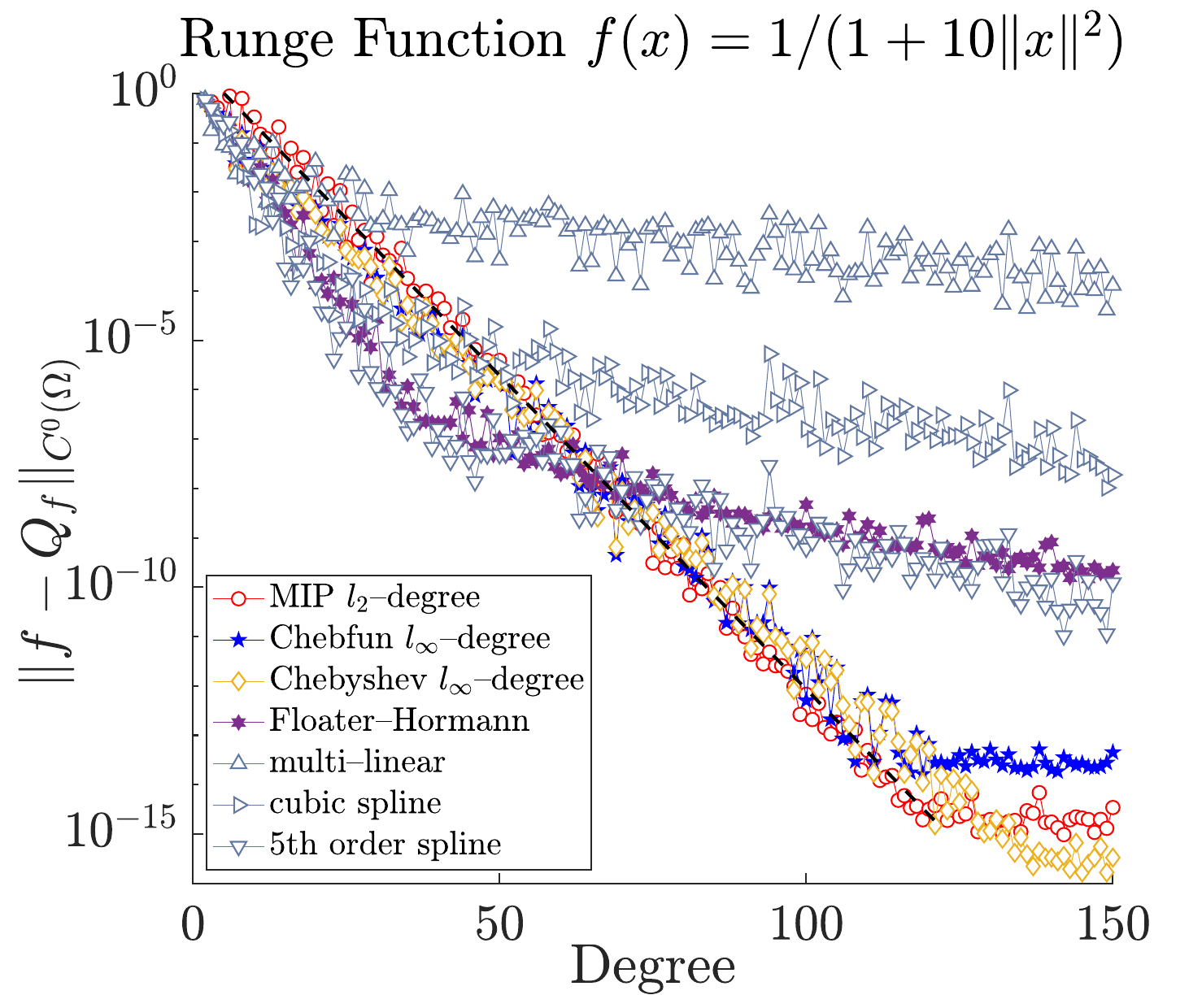}
  \includegraphics[scale=0.25]{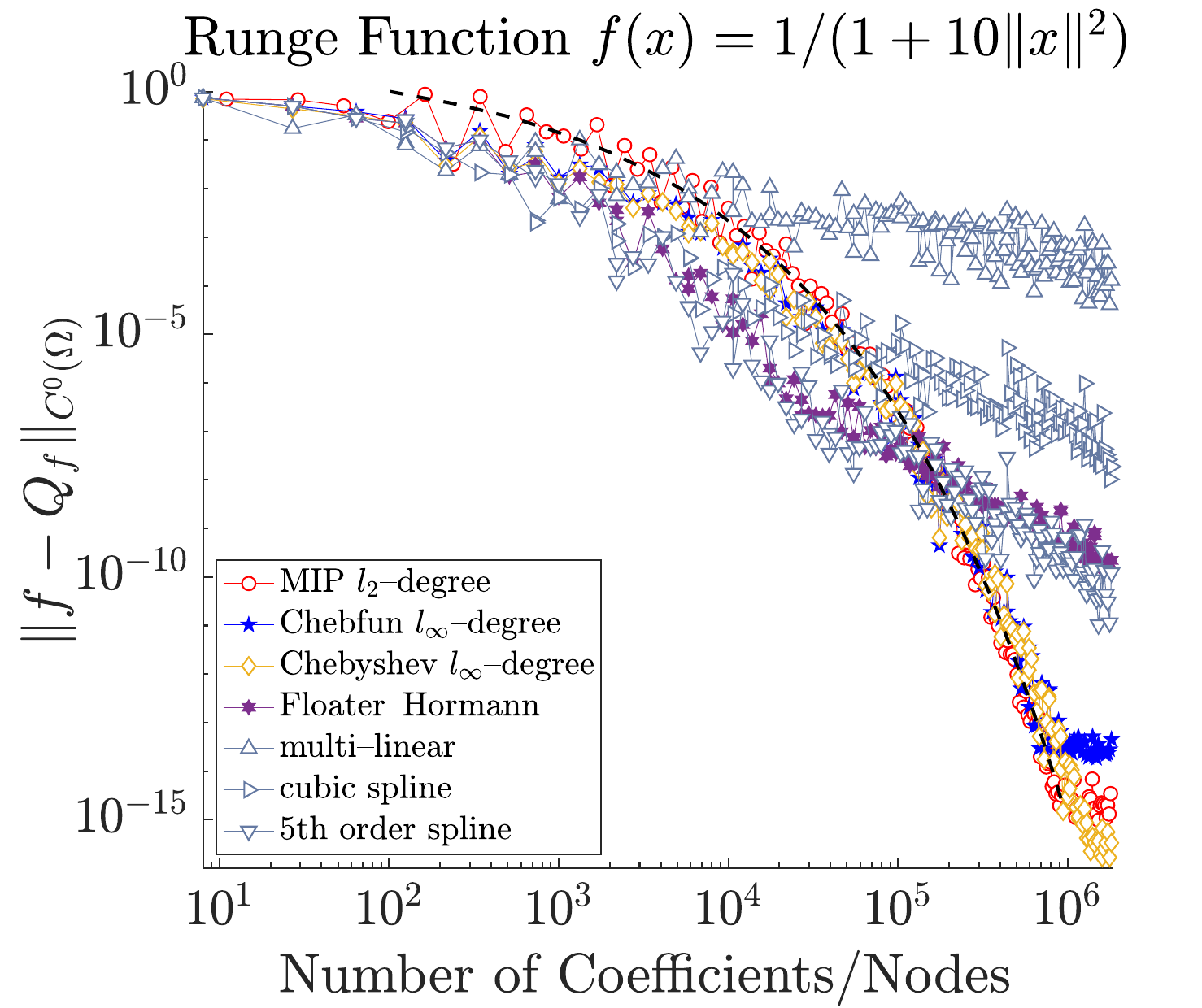}
\caption{Approximation errors for the benchmarked methods interpolating the Runge function in dimension $m=3$. } \label{Rate_3D}
\end{figure}

Figure \ref{Rate_3D} shows the results of the same experiment in dimension $m=3$, leaving out the infeasible methods.
The observations made in 2D remain valid. However, \emph{Floater-Hormann} becomes indistinguishable from \emph{$5^{th}$-order splines}.
Further, when considering the amount of coefficients/nodes required to determine the interpolant, plotted in the right panel (with logarithmic scales on both axes). The polynomial convergence rates of \emph{Floater-Hormann} and all
\emph{spline-type} approaches become visible.
\emph{MIP} requires  $122^3/899028 \approx 2$-times less coefficients/nodes than \emph{Chebyshev} or
\emph{Chebfun} to approximate $f$ to machine precision for $n=121$.

% Figure \ref{Rate_4D} shows the results for dimension $m=4$. Spline interpolation was not able to scale to high degrees due to computer memory requirements.
% To simulate a rescaling of the hypercube to $\frac{1}{\sqrt{10}} \Omega = [-\frac{1}{\sqrt{10}},\frac{1}{\sqrt{10}}]^m$, we approximate the Runge function on two scales: Once for Runge factor $RF =10$ and once for $RF=1$.
% The results for $RF=10$ show a similar situation as the results in 3D, but \emph{Chebyshev} and \emph{MIP} provide no advantage any more over \emph{Floater-Hormann interpolation}.
% For $RF=1$, only \emph{Chebyshev} and \emph{MIP} converge down to machine precision. But \emph{MIP} reaches that goal earlier ($n=40/47$) than  \emph{Chebyshev}, and with less interpolation nodes
% $ \frac{|C_{\mathrm{Chebyshev}}|}{|C_{\mathrm{MIP}}|} = \frac{5308416}{858463} \approx 6$.

Figure \ref{Rate_4D} shows the results for dimension $m=4$. Spline interpolation was not able to scale to high degrees due to computer memory requirements.
To simulate the behavior for higher degrees we rescale the hypercube to $\frac{1}{\sqrt{10}} \Omega = [-\frac{1}{\sqrt{10}},\frac{1}{\sqrt{10}}]^m$. That is, we approximate the Runge function on two scales:
Once for Runge factor $RF =10$ and once for $RF=1$.
The results for $RF=10$ show a similar situation as the results in 3D.
However, Figure \ref{Rate_3D} suggests that again degree $n \approx 75$ is the
crossover point, where \emph{MIP} and \emph{Chebyshev} become the superior approaches. Indeed,
for $RF=1$, only \emph{Chebyshev} and \emph{MIP} converge down to machine precision.
But \emph{MIP} reaches that goal earlier ($n=40/47$) than  \emph{Chebyshev}, and with less interpolation nodes
$ \frac{|C_{\mathrm{Chebyshev}}|}{|C_{\mathrm{MIP}}|} = \frac{5308416}{858463} \approx 6$.

\begin{figure}[t!]
\includegraphics[scale=0.25]{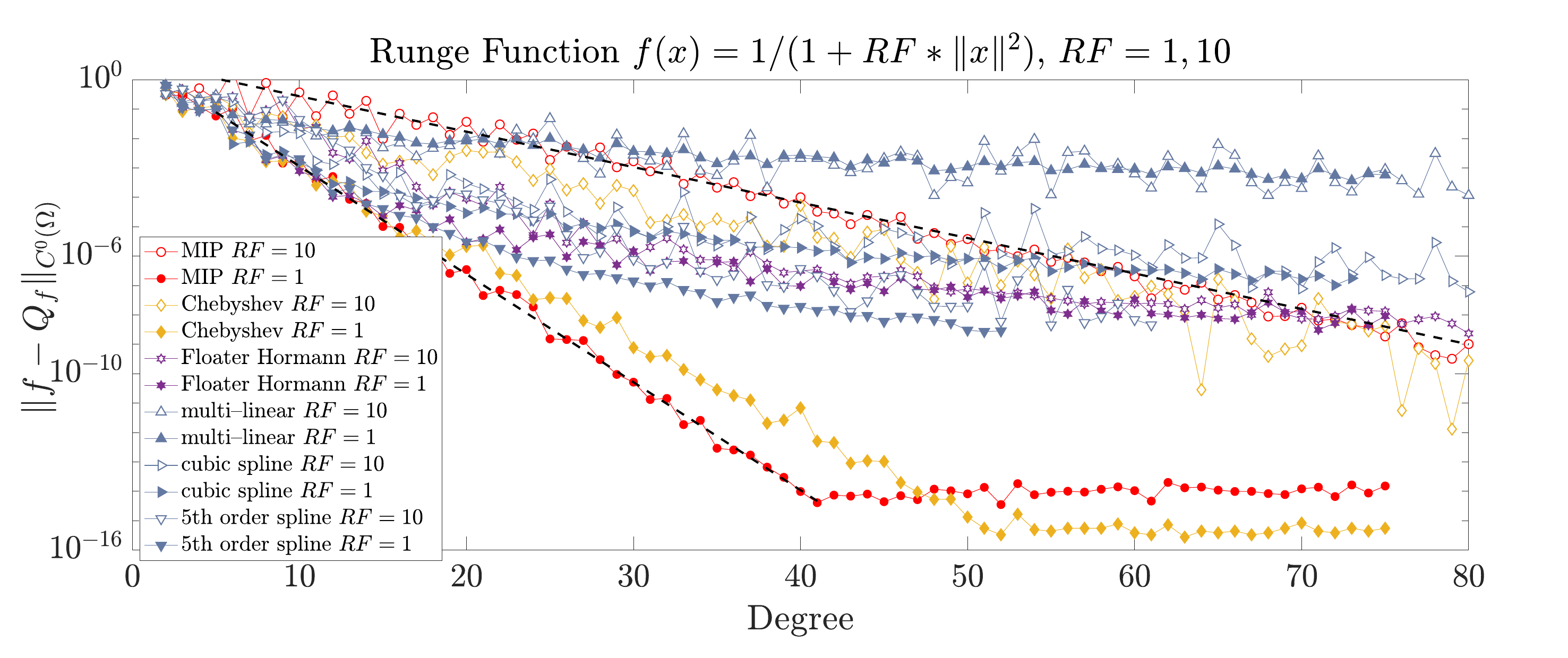}
\vspace{-0.5cm}
\caption{Approximation errors for the benchmarked methods interpolating the Runge function in dimension $m=4$. } \label{Rate_4D}
\end{figure}
\begin{figure}[t!]
  \includegraphics[scale=0.25]{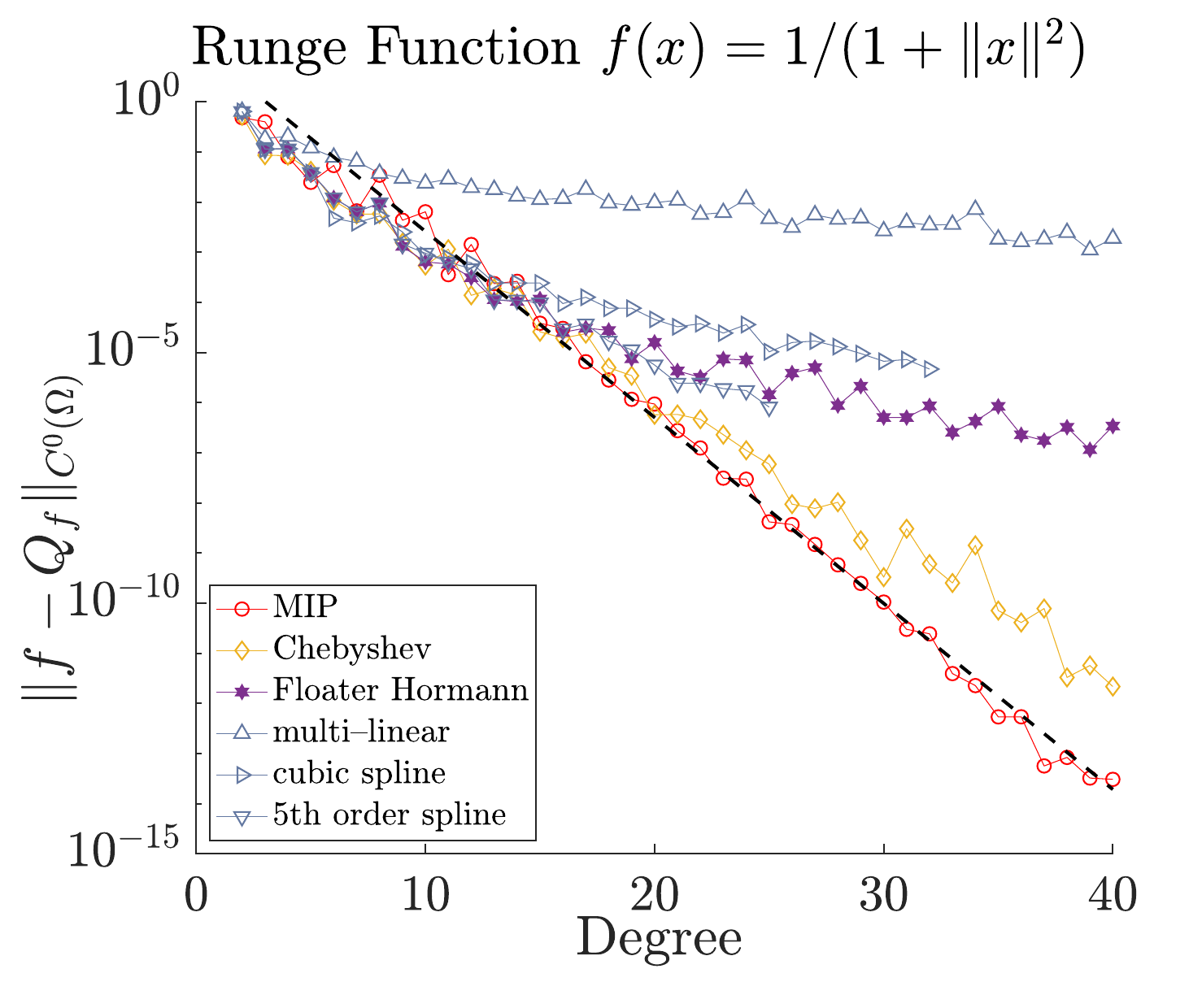}
  \includegraphics[scale=0.25]{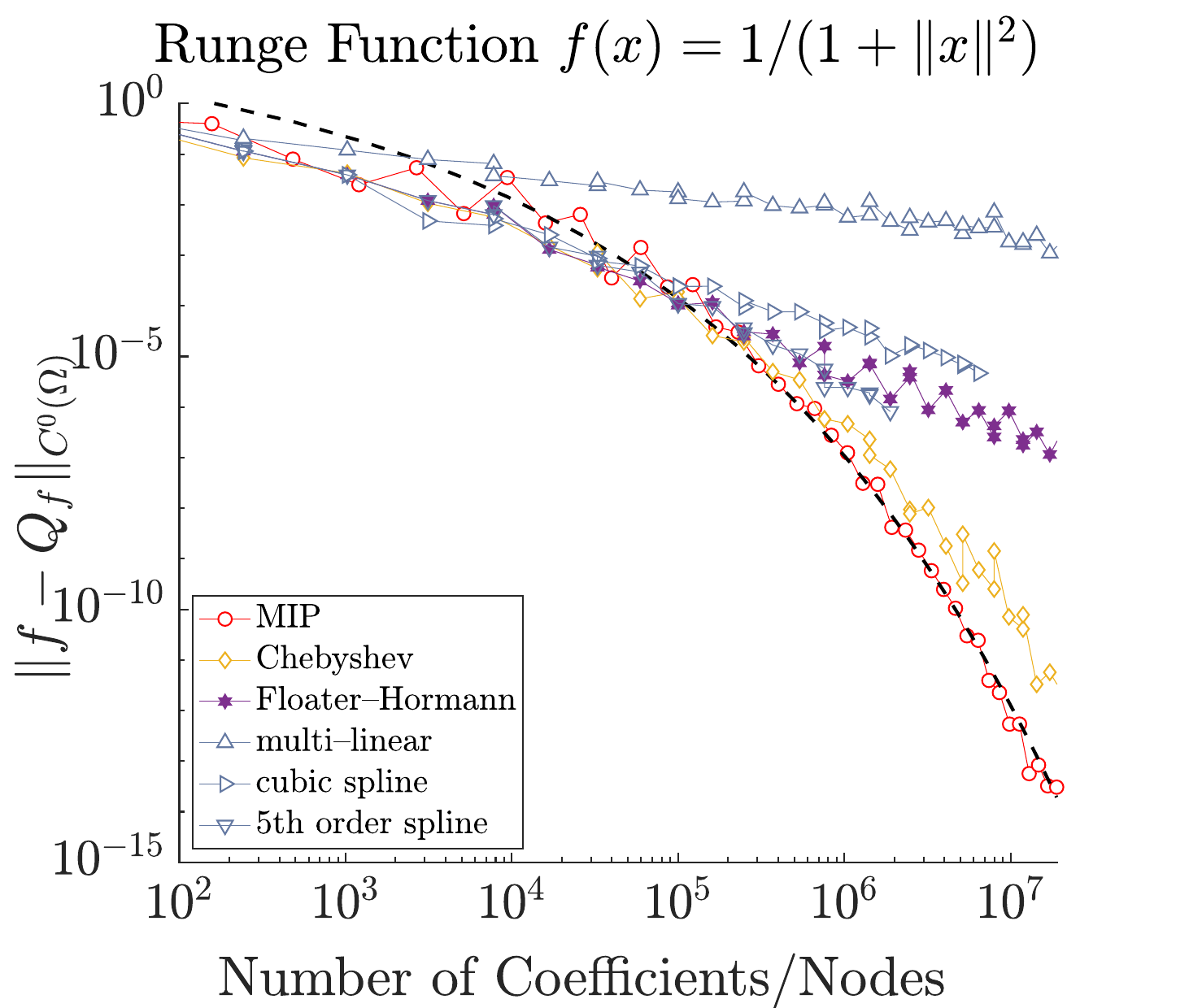}
\caption{Approximation errors for the benchmarked methods interpolating the Runge function in dimension $m=5$. } \label{Rate_5D}
\end{figure}

The same is true in dimension $m=5$, as Fig.~\ref{Rate_5D} illustrates. Especially when considering the right plot (with logarithmic scales on both axes), we observe that
\emph{MIP} best resists the curse of dimensionality by yielding 2 orders of magnitude better accuracy than \emph{Chebyshev} for $n=40$  ($3.0\cdot 10^{-14}$ vs.~$2.1\cdot 10^{-12}$) with less interpolation nodes
$ \frac{|C_{\mathrm{Chebyshev}}|}{|C_{\mathrm{MIP}}|} = \frac{115856201}{18920038} \approx 6$.

To assess the convergence rates, we fit the data for \emph{MIP} with the model $y = c\rho^{-n}$ with an $R$-squared of $0.99$ or better, indicated by the dashed lines in the corresponding figures.
This yields the exponential decays reported in Table~1 for the approximation errors in the corresponding ranges.
\begin{table}[t!] \label{Trate}
  \begin{tabular}{ccccccc}
dim & fitting range  & $\rho_{RF=10}$ & $c_{RF =10}$ &  fitting range & $\rho_{RF=1}$ & $c_{RF =1}$   \\
 \hline
 2 & $ 2 \sim 121$  & $1.35$ & $4.30$  & & & \\
 3&  $2 \sim 121$ & $1.34$  & $4.41$   & & &\\
 4 & $ 2 \sim 80$  & $1.32$ & $4.42$ & $ 2 \sim 40$  &2.33 & 5.40\\
 5 &  & & & $ 2 \sim 40$ & 2.35 & 13.37
\end{tabular}
\caption{Fitted convergence rates of \emph{MIP}.}
\vspace{-0.5cm}
\end{table}
The Trefethen radius  is given by $\rho_\text{max} = 1 + \sqrt{2} \approx 2.41$ for $RF =1$. Thus, the upper bounds in Theorem \ref{theorem:TREF} are almost achieved by the MIP method. In contrast, \emph{Chebyshev}  just reaches convergence rates
$\rho_{RF=1} \approx 1.9$ in dimension $m=4,5$.

In summary, Experiment 1 confirms that \emph{MIP} converges as expected from Theorem \ref{theorem:TREF}.
Compared to the other methods tested, \emph{MIP} is  efficient in reaching machine precision.
\emph{MIP} also seems to resistant to the curse of dimensionality best, which becomes increasingly visible in higher dimensions, thus supporting the prediction of Remark \ref{rem:eff}.

\subsection{Regression  on curved manifolds}\label{sec:NUM2}

We realize the algorithm of Carl de Boor and Amon Ros \cite{deBoor2,deBoor} in terms of Corollary \ref{cor:Torus}  in case of the torus $M=\T^2_{R,r}$.  That is, we consider
$$ Q_{\T^2_{R,r}}(x,y,z) = \left(x^{2}+y^{2}+z^{2}+R^{2}-r^{2}\right)^{2}-4R^{2}\left(x^{2}+y^{2}\right)$$
with $R=0.7$ and $r=0.3$. $\T^2_{R,r} = Q_{\T^2_{R,r}}^{-1}(0)$ is an algebraic hypersurface of degree 4. Given a function $f : \Omega \lo \R$,
we aim
to interpolate the restriction $f_{|\T^2_{R,r}}$.

\begin{experiment}\label{exp:torus} We choose $A=A_{m,n,p}$ with $m=3$, $p=2$, and sample
$S=\lfloor 1,5\cdot |A|\rfloor \in \N$ uniformly random nodes $P_{\T^2_{R,r}}$ on the torus $\T^2_{R,r}$, as illustrated in Figure \ref{Torus} (left). Further, we generate $P_A$ with respect to
$$\mathrm{GP} =  \Cheb_n^{0} \oplus   \Cheb_n^{0}  \oplus  r\cdot \Cheb_n^{0} \,,$$
yielding a feasible grid near the sampled nodes $P_{\T^2_{R,r}}$. Since $\T^2_{R,r}$ is a hypersurface of degree $4$, the nodes $p\in P_{\T^2}$ are not unisolvent for $\Pi_{A_{m,n,p}}$ with $n\geq 4$. Thus, due to Corollary \ref{cor:Torus}
$$m_M:= \{Q \in \Pi_A \mi Q_{|{M}} \equiv 0\} \not = \{0\}\,, \quad M = \T^2_{R,r}\,.$$
However, by Remark \ref{PROP1} and Theorem \ref{theorem:Dual} the splitting
$$\Pi_A \cong \Pi_A \big / m_{M}  + m_M   =: \Pi_{A,\T^2}  +  \Pi_{A,\T^2}^\perp $$
holds with probability $1$.
For $A=A_{3,4,2}$, we have $\dim \Pi_{A,\T^2}^\perp =1$ and by computing a basis $\mu \in \Pi_A$ of $\dim \Pi_{A,\T^2}^\perp$  as in Eq.~\eqref{kernel} we have determined a level-set function
$Q_{\T^2_{R,r}}=\mu$, i.e.,  $\T^2_{R,r} = Q_{\T^2_{R,r}}^{-1}(0)$.

Further, considering $R_A =(L_\alpha(p_i))_{(i,\alpha) \in S \times A }$, where $L_\alpha$, $\alpha \in A$ denote the Lagrange polynomials w.r.t.~$P_A$, we find the Lagrange coefficients $C_{\mathrm{Lag}}$
of the interpolant $Q_{f,A | \T^2_{R,r}}$  of $f_{|\T^2_{R,r}}$ for any $f : \Omega \lo\R$  by solving
\begin{equation*}
 R_AC_{\mathrm{lag}} \approx F\,,\quad  F =(f(p_1),\dots,f(p_{S})) \in \R^S\,, \quad p_i \in P_{\T^2_{R,r}}\,, \,\,\,1 \leq i \leq S
\end{equation*}
using standard MATLAB least-square regression.
\end{experiment}

\begin{figure}[t!]
\includegraphics[scale=1]{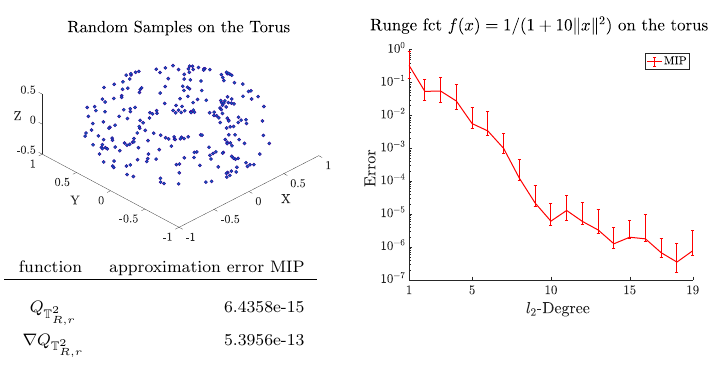}
\vspace{-0.25cm}
\caption{Uniformly random nodes on the torus $\T^2_{R,r}$ with $R=0.7$ and $r=0.3$ (left), approximation errors of the level set function $Q_{\T^2_{R,r}}$ and its gradient $\nabla Q_{\T^2_{R,r}}$ (table) and
approximation error for the restriction $f_{|\T^2_{R,r}}$ of Runge function to the torus (right).}
 \label{Torus}
\end{figure}

The results of this experiment are shown in Fig.~\ref{Torus}. All errors  are measured on the $S$ nodes $P_{\T^2_{R,r}}$, plus additional 200 uniform at random sampled points $P \subseteq \T^2_{R,r}$
for 10 independent repetitions in each case.
The inset table in Fig.~\ref{Torus} shows that the approximation errors for the level-set function $Q_{\T^2_{R,r}}$ and its gradient $\nabla Q_{\T^2_{R,r}}$, corresponding to the surface normal, are within machine precision.
The right panel of Fig.~\ref{Torus} shows the mean approximation errors  with min--max error bars for interpolating the restriction $f_{R|\T^2_{R,r}}$ of the Runge function  $f(x) = 1/(1+10\|x\|^2)$ to the torus,
measured analogously.

The results suggest that the notion of the matrix $R_A$ defined with respect to the unisolvent grid $P_A$ is the main reason for the fast approximation rate on the Runge function.
Implicitly, we therefore also validate the numerical stability of the pre-computation approach.

We are not aware of any other approach that can handle such an interpolation task without requiring a \emph{triangulation} or \emph{parameterization} of the manifold $\T^2_{R,r}$.
Therefore, these results demonstrate the flexibility of \emph{MIP} and suggest its use in numerical methods on curved surfaces, including quadrature schemes, ODE and PDE solvers,  and optimization algorithms.

\section{Conclusion}

We have generalized the notion of unisolvent nodes for polynomial interpolation in arbitrary dimensions with respect to a generalized concept of polynomial degree and for non-tensorial node distributions.
This allowed us to extend the classic 1D Newton and Lagrange interpolation methods to multivariate
schemes in a numerically stable and efficient way, resulting in a practically implemented algorithm with $\Oc(|A|^2)$ runtime complexity and $\Oc(|A|)$ memory complexity.
We also provided the theory for generalizing the approach to interpolation of scattered data on curved manifolds.

We empirically observed that the resulting algorithm reaches the optimal approximation rate given by Lloyd N.~Trefethen's Theorem \cite{Lloyd2}, supporting the conjecture that these rates apply in general \cite{converse}.
In contrast to previous approaches, such as Chebfun \cite{chebfun}, multivariate splines \cite{Boor:SP}, and Floater-Hormann interpolation \cite{floater}, the present
\emph{MIP} algorithm achieves exponential approximation rates for the Runge function using only sub-exponentially many interpolation nodes.
This suggests that we have found an efficient approximation scheme that overcomes the curse of dimensionality for a generic class of functions.

In closing, we discuss related concepts and possible further developments.

\subsection{Generalizing the notion of unisolvence}
We generalized the notion of unisolvence beyond the pioneering work of Kuntzmann, Guenther, and Roetman~\cite{Guenther,kuntz}. However, an even more general notion should be possible by
considering (more) general
hypersurfaces $\widetilde H$ instead of the linear hyperplanes $H \subseteq \R^m$ used in Theorem \ref{theorem:UN}. This might also enable the native construction of unisolvent nodes on general level-set manifolds
and to derive a Newton interpolation scheme on curved spaces.

\subsection{Barycentric Lagrange interpolation}\label{sec:Bary}

In 1D, barycentric Lagrange interpolation is the most efficient interpolation scheme \cite{berrut} for fixed nodes. Both
determining the interpolant $Q_{f,n}$ and evaluating $Q_{f,n}$ at any given $x_0 \in \R$ require linear time $\Oc(n)$.
This is achieved by precomputing the constant \emph{barycentric weights} that only depend on the locations of the nodes, but not on the function $f$.

We have already established preliminary theoretical results towards generalizing this approach to $m$D for the case of $l_1$-degree~\cite{sivkin}.
% Our generalization is based on the observation that the transformation matrices $\mathrm{NL}_A$, $\mathrm{LN}_A$, $\mathrm{CN}_A$, $\mathrm{NC}_A$
% from Theorem \ref{theorem:Lag} are spare, but structured. This structure has not been exploited so far. It is known that some structured matrices can be inverted and multiplied much faster than the general case~\cite{struct2,struct1}.
This suggests that the current complexity  $\Oc(|A|^2)$ of interpolation and evaluation for the case of multi-indices $A =A_{m,n,p}$, $p>1$, can be further reduced.
%In practice, efficient interpolation in dimensions $m \geq 6$ might thus become possible, as it is important when considering phase spaces of 3D dynamical systems.

\subsection{Multivariate polynomial regression} \label{sec:GEO}

For a given function $f : \Omega \lo \R$ and a set of nodes $P\subseteq \Omega =[-1,1]^m$, $m\in\N$, we consider the graph $G = f(P) \times P \subseteq  \R^{m+1}$. Corollary \ref{cor:Torus} and Experiment \ref{exp:torus}
allow identifying the hypersurface $M\supseteq G$ that contains $G$ as a level set of polynomials $Q_{M,i} \in \Pi_{m+1}$, $M = \cap_{i=1}^{n+1-\dim M} Q_{M,i}^{-1}(0)$
and to interpolate the restriction $g_{|M}$ of any function $g :  \R^{m+1} \lo \R$ to $M$.
We consider both aspects crucial steps towards developing a multivariate polynomial regression scheme,
opening up applications in multivariate analysis, statistics, and topological analysis \cite{anderson1958introduction,persitant_SURVEY,friedman2001elements,mardia1979,mardiakent,shestopaloff2017new,persitant}.

\subsection{Trigonometric interpolation}
In 1D, \emph{trigonometric Clairaut--Lagrange--} \linebreak \emph{Gauss interpolation} can be used to compute the discrete Fourier transform (DFT) of a periodic function $f$ \cite{gaussFFT2}.
This fact was for example used in the development of the famous \emph{Cooley-Tukey algorithm} \cite{cooley}, which re-instantiated an algorithm by Carl F.~Gauss \cite{gaussFFT}
to yield a modern realization of the \emph{Fast Fourier Transform}.
These concepts are also closely related to the invention of \emph{wavelets} \cite{strang1994wavelets}. Revisiting these aspects from the perspective of
multivariate Lagrange interpolation might be worthwhile to make progress  in open problems, such as Fourier transformation of non-periodic, highly oscillating signals or fast Fourier transforms on scattered nodes \cite{barnett2019parallel,greengard2004accelerating,IEEE}.

\subsection{Numerical integration}
Until today, the classic \emph{Gauss quadrature} formula is the best approach to approximating integrals $I_{\mathrm{Gauss}}(f) \approx \int_\Omega f(x) \, \mathrm{d}x$ in one variable~\cite{gaussQUAD,jacobi}.
Many contributions toward extending this approach to higher dimensions have been made \cite{cools2002,cools,hammer1956numerical,stroud}.
This list is by no means exhaustive, and research in this direction is actively ongoing. The present notion of unisolvent nodes might be helpful in this endeavor.

\section*{Acknowledgements}
Leslie Greengard, Christian L.~Mueller, Alex Barnett, Manas Rachh, Heide Meissner,  Uwe Hernandez Acosta, and Nico Hoffmann are deeply acknowledged for their inspiring hints and helpful discussions.
Further, we are grateful to Michael Bussmann and thank the whole CASUS institute (G\"{o}rlitz, Germany) for hosting stimulating workshops on the subject.

\bibliographystyle{amsplain}
\bibliography{Ref.bib}   % name your BibTeX data base

\end{document}